\documentclass[a4paper,
               10pt,
               pdftex,
               normalheadings,
               headsepline,
               footsepline,
               headinclude,
               footinclude,
               DIV14,
               abstracton]
{scrartcl}

\usepackage
{
    graphicx,
    amssymb,
    amsmath,
    amsthm,
    xcolor,
    dsfont,
    algpseudocode,
    authblk,
}

\usepackage[ngerman, english]{babel}

\usepackage[colorlinks,
            pdffitwindow=false,
            plainpages=false,
            pdfpagelabels=true,
            pdfpagemode=UseOutlines,
            pdfpagelayout=SinglePage,
            bookmarks=false,
            colorlinks=true,
            hyperfootnotes=false,
            linkcolor=blue,
            citecolor=green!50!black]
{hyperref}

\usepackage{enumerate}
\usepackage{algorithm}

\usepackage{tabularx}

\usepackage[left=2.50cm, right=2.50cm, top=2.00cm, bottom=3.50cm]{geometry}
\usepackage{sectsty}
\sectionfont{\fontsize{10}{10}\selectfont}
\subsectionfont{\fontsize{10}{10}\selectfont}

\newcommand{\R}{\mathbb{R}}
\newcommand{\N}{\mathbb{N}}

\newcommand{\SetU}{\mathcal{U}}
\newcommand{\SetX}{\mathcal{X}}

\newcommand{\PS}{\mathcal{P}}

\newcommand{\PF}{\mathcal{P}_{{F}}}

\newtheorem{theorem}{Theorem}[section]
\newtheorem{corollary}[theorem]{Corollary}

\newtheorem{proposition}[theorem]{Proposition}
\newtheorem{definition}[theorem]{Definition}
\newtheorem{remark}[theorem]{Remark}
\newtheorem{example}[theorem]{Example}

\let\OLDthebibliography\thebibliography
\renewcommand\thebibliography[1]{
	\OLDthebibliography{#1}
	\setlength{\parskip}{0pt}
	\setlength{\itemsep}{0pt plus 0.3ex}
}

\begin{document}

\title{Explicit multiobjective model predictive control for nonlinear systems with symmetries}
\author[1]{Sina Ober-Bl{\"o}baum}
\author[2]{Sebastian Peitz}
\affil[1]{\normalsize Department of Engineering Science, University of Oxford, UK}
\affil[2]{\normalsize Department of Mathematics, Paderborn University, Germany}

\date{}

\maketitle

\begin{abstract}
	Model predictive control is a prominent approach to construct a feedback control loop for dynamical systems. Due to real-time constraints, the major challenge in MPC is to solve model-based optimal control problems in a very short amount of time. For linear-quadratic problems, Bemporad et al.~have proposed an explicit formulation where the underlying optimization problems are solved a priori in an offline phase. In this article, we present an extension of this concept in two significant ways. We consider nonlinear problems and -- more importantly -- problems with multiple conflicting objective functions. In the offline phase, we build a library of Pareto optimal solutions from which we then obtain a valid compromise solution in the online phase according to a decision maker's preference. Since the standard multi-parametric programming approach is no longer valid in this situation, we instead use interpolation between different entries of the library. To reduce the number of problems that have to be solved in the offline phase, we exploit symmetries in the dynamical system and the corresponding multiobjective optimal control problem. The results are verified using two different examples from autonomous driving.
\end{abstract}

\section{Introduction}
\label{sec:Introduction}
In many applications from industry and economy, several criteria are of equal interest. Popular examples include production processes, where we want to maximize the quality while minimizing the production cost, or transportation, where the objectives are fast and energy efficient driving. Since these objectives are in general contradictory, the solution consists of the set of \emph{optimal compromises} -- the so-called \emph{Pareto set} -- instead of a single optimum, and a compromise can be selected interactively according to a \emph{decision maker's} preference. This way, the flexibility in operating a complex system can be significantly increased. 

Regardless of the solution method (cf.~\cite{Ehr05} for an introduction to deterministic and \cite{CLV07} for evolutionary methods), the computation of the entire Pareto set is infeasible in the real-time context, i.e., in a \emph{model predictive control (MPC)} framework \cite{GP17}. There exist several approaches to circumvent this dilemma, see \cite{PD18} for a survey. One is a priori scalarization, where the conflicting objectives are synthesized into a scalar objective using, e.g., weighted sums \cite{BP09a} or reference point techniques \cite{ZFT12}. A different approach is to stop the expensive computation prematurely and use non-converged solutions \cite{LBK08,GGG+12}. Another frequently applied way to incorporate multiple objectives is to use a classical feedback controller and optimize the controller parameters with respect to several criteria \cite{KWTD11,HNS+13}. An additional criterion for selecting the objectives such that stability is preserved was studied in \cite{GS17}. Finally, an approach has recently been presented in \cite{PSOB+17} in the context of autonomous driving which is motivated by \emph{explicit MPC} \cite{BMDP02}. Here, the solution is determined in an offline phase such that the online phase is reduced to selecting the optimal control from a library. Since the problem under consideration is nonlinear, one cannot use the multi-parametric programming approach proposed in \cite{BMDP02}. Instead, a numerical grid is introduced for the relevant parameters and a \emph{multiobjective optimal control problem (MOCP)} is solved for each grid point. In order to reduce the number of parameters, symmetries in the problem can be exploited using ideas from motion planning with motion primitives \cite{Frazzoli2001,FrDaFe05,Kob08,FOK10}. In the online phase, linear interpolation between the neighboring entries in the library is performed to quickly obtain the relevant Pareto set. According to the decision maker's preference, a Pareto optimal control is then applied to the plant.

In this article, we present a detailed analysis of the algorithm developed in \cite{PSOB+17} and study under which conditions the MOCP possesses symmetries that can be exploited to reduce the number of relevant parameters. We will see that these conditions are reduced to already known ones in the case of linear-quadratic, single-objective problems \cite{DB12}. Both analytic as well as numerical approaches to identify symmetries are discussed. We then present the explicit multiobjective MPC algorithm and give a detailed description of the online and the offline phase, including an automated procedure for solving a large number of MOCPs in parallel. Two complex examples are used to study the above-mentioned properties and to analyze the performance of the proposed method. In the first one, a race track is considered where the goals are to drive as fast and safe as possible. In the second example, we want to control the longitudinal dynamics of an electric vehicle with respect to fast and energy efficient driving (see also \cite{PSOB+17}).

The remainder of this article is structured as follows. In Section~\ref{sec:MOMPC}, we introduce the multiobjective optimal control problem, the basic definitions for multiobjective optimization and the MPC framework we are using. The exploitation of symmetries is discussed in Section~\ref{sec:Symmetries} and the nonlinear explicit MPC algorithm is presented in Section~\ref{sec:Algorithm}. The two examples from autonomous driving are covered in detail in Section~\ref{sec:Examples} and finally, a conclusion is drawn in Section~\ref{sec:Conclusion}.

\section{Multiobjective optimization and model predictive control}
\label{sec:MOMPC}
In this section, we introduce the general nonlinear multiobjective optimal control problem (MOCP) and introduce the basic notions of multiobjective optimization. We then give a very short introduction to MPC.

\subsection{Multiobjective optimization and optimal control}
\label{subsec:MOMPC_Multiobjective}
The goal in multiobjective optimal control is the minimization of multiple conflicting objectives while taking the system dynamics (here described by an ordinary differential equation) into account:
\begin{subequations}\label{eq:MOCPx}
\begin{align}
    \min_{x \in \SetX, u \in \SetU} J(x,u)
    &= \left( \begin{array}{c}
    \int_{t_0}^{t_e} C_1(x(t),u(t))~dt + \Phi_1(x(t_e)) \\ \vdots \\ \int_{t_0}^{t_e} C_k(x(t),u(t))~dt + \Phi_k(x(t_e))
    \end{array}\right) \label{eq:MOCPxJ}
    \\
    \mbox{s.t.} \qquad \qquad \dot{x}(t) &= f(x(t),u(t)),\qquad \qquad t \in (t_0,t_e],  \label{eq:MOCPxd}\\
    x(t_0) &= x_0, \label{eq:MOCPx0} \\
    g_i(x(t),u(t)) &\leq 0, \quad i=1,\ldots,l, \qquad ~ ~ t \in (t_0,t_e],  \label{eq:MOCPg}\\
    h_j(x(t),u(t)) &= 0, \quad j=1,\ldots,m, \qquad t \in (t_0,t_e], \label{eq:MOCPh}
\end{align}
\end{subequations}
where $x(t)\in\R^{n_x}$ is the system state and $u(t)\in U \subset \R^{n_u}$ is the control variable with $U$ being closed and convex.
Since we consider optimization problems over state and control trajectories $x(t)$ and $u(t)$, the sets $\SetX = W^{1,\infty}([t_0,t_e], \R^{n_x})$ and $\SetU = L^{\infty}([t_0,t_e], U)$ are function spaces.
$J: \SetX \times \SetU \rightarrow \R^k$ is the cost functional with $k$ conflicting objectives and with continuously differentiable functions $C_i: \R^{n_x} \times U \rightarrow \R$, $\Phi_i: \R^{n_x} \rightarrow \R$. 
Furthermore, $f: \R^{n_x} \times U \rightarrow \R^{n_x}$ is Lipschitz continuous, and $g: \R^{n_x} \times U \rightarrow \R^l$, $g = (g_1,\ldots,g_l)^{\top}$ and $h: \R^{n_x} \times U  \rightarrow \R^m$, $h = (h_1,\ldots,h_m)^{\top}$, are continuously differentiable inequality and equality constraint functions, respectively. $(x,u)$ is called \emph{feasible pair} if it satisfies the constraints \eqref{eq:MOCPxd}--\eqref{eq:MOCPh}. The space of the control trajectories $\SetU$ is also called the \emph{decision space} and the image of all feasible pairs forms the \emph{objective space}. 

Problem \eqref{eq:MOCPx} can be simplified by introducing the flow of the dynamical system:
\[
    \varphi_u(x_0,t) = x_0 + \int_{t_0}^t f(x(t),u(t))\, dt.
\]
This way, the explicit dependency of $J$, $g$ and $h$ on $x$ can be removed:
\begin{equation}\label{eq:MOCP}\tag{MOCP} 
	\begin{aligned}
	    &\min_{u \in \SetU} \hat{J}(x_0,u) 
	    \\
	    \hat{g}_i(x_0,\mathfrak{u}) &\leq 0, \quad i=1,\ldots,l, \qquad ~ ~ t \in (t_0,t_e],  \\
	    \hat{h}_j(x_0,\mathfrak{u}) &= 0, \quad j=1,\ldots,m, \qquad t \in (t_0,t_e], 
	\end{aligned}
\end{equation}
where
\[
    \hat{J}_i(x_0,u) = \int_{t_0}^{t_e} \hat{C}_i(x_0,\mathfrak{u})\,dt + \hat{\Phi}_i(x_0,\mathfrak{u})
\]
with $\hat{C}_i(x_0,\mathfrak{u}) := C_i(\varphi_u(x_0,t),u(t))$ and $\hat{\Phi}_i(x_0,\mathfrak{u}) := \Phi_i(\varphi_u(x_0,t_e))$ for $i = 1,\ldots,k$. Here, $\mathfrak{u} := u|_{[t_0,t]}$ is introduced to preserve the time dependency. The constraints $\hat{g}(x_0,\mathfrak{u})$ and $\hat{h}(x_0,\mathfrak{u})$ are defined accordingly. $u$ is called a \textit{feasible curve} if it satisfies the equality and inequality constraints $\hat{g}_i,i=1,\ldots,l$, and $\hat{h}_j,j=1,\ldots,m$.

In contrast to single objective optimization problems, there exists no total order of the objective function values in $\R^k$ with $k \geq 2$. Thus, we introduce the following partial order:
\begin{definition}\label{def:PartialOrder}
    Let $v, w \in \R^k$. The vector $v$ is \emph{less than} $w$ (denoted by $v < w$), if $v_i < w_i$ for all $i \in \left\lbrace 1, \ldots, k \right\rbrace$. The relation $\leq$ is defined in an analogous way. 
\end{definition}

Since there is no total order, we cannot expect to find isolated optimal curves for \eqref{eq:MOCP}. Instead, the solution is the \emph{set of optimal compromises} (also called the \emph{Pareto set} or \emph{set of non-dominated curves}):

\begin{definition}\label{def:Nondominance}
    Consider the multiobjective optimization problem \eqref{eq:MOCP}. Then
    \begin{enumerate}
        \item a feasible curve $u^*$ \emph{dominates} a curve $u$, if $\hat{J}(x_0,u^*) \leq \hat{J}(x_0,u)$ and $\hat{J}(x_0,u^*) \neq \hat{J}(x_0,u)$.
        \item a feasible curve $u^*$ is called \emph{globally Pareto optimal} if there exists no feasible curve $u \in \SetU$ dominating $u^*$. The image $\hat{J}(x_0,u^*)$ of a globally Pareto optimal curve $u^*$ is called a \emph{globally Pareto optimal value}. If this property holds in a neighborhood $U(u^*) \subset \SetU$, then $u^*$ is called \emph{locally Pareto optimal}.
        \item the set of non-dominated feasible curves is called the \emph{Pareto set} $\PS$, its image the \emph{Pareto front} $\PF$.
    \end{enumerate}
\end{definition}
\noindent Definition~\ref{def:Nondominance} is visualized in Figure~\ref{fig:MOP_Example} for a finite-dimensional problem, i.e.~curves reduce to points in the Euclidean space. We see that for each point that is contained in the Pareto set (the red line in \ref{fig:MOP_Example}~(a)), one can only improve one objective by accepting a trade-off in at least one other objective, see the red line in Figure~\ref{fig:MOP_Example}~(b).
\begin{figure}[h!]
    \centering
    \parbox[b]{0.24\textwidth}{\centering \includegraphics[width=0.3\textwidth]{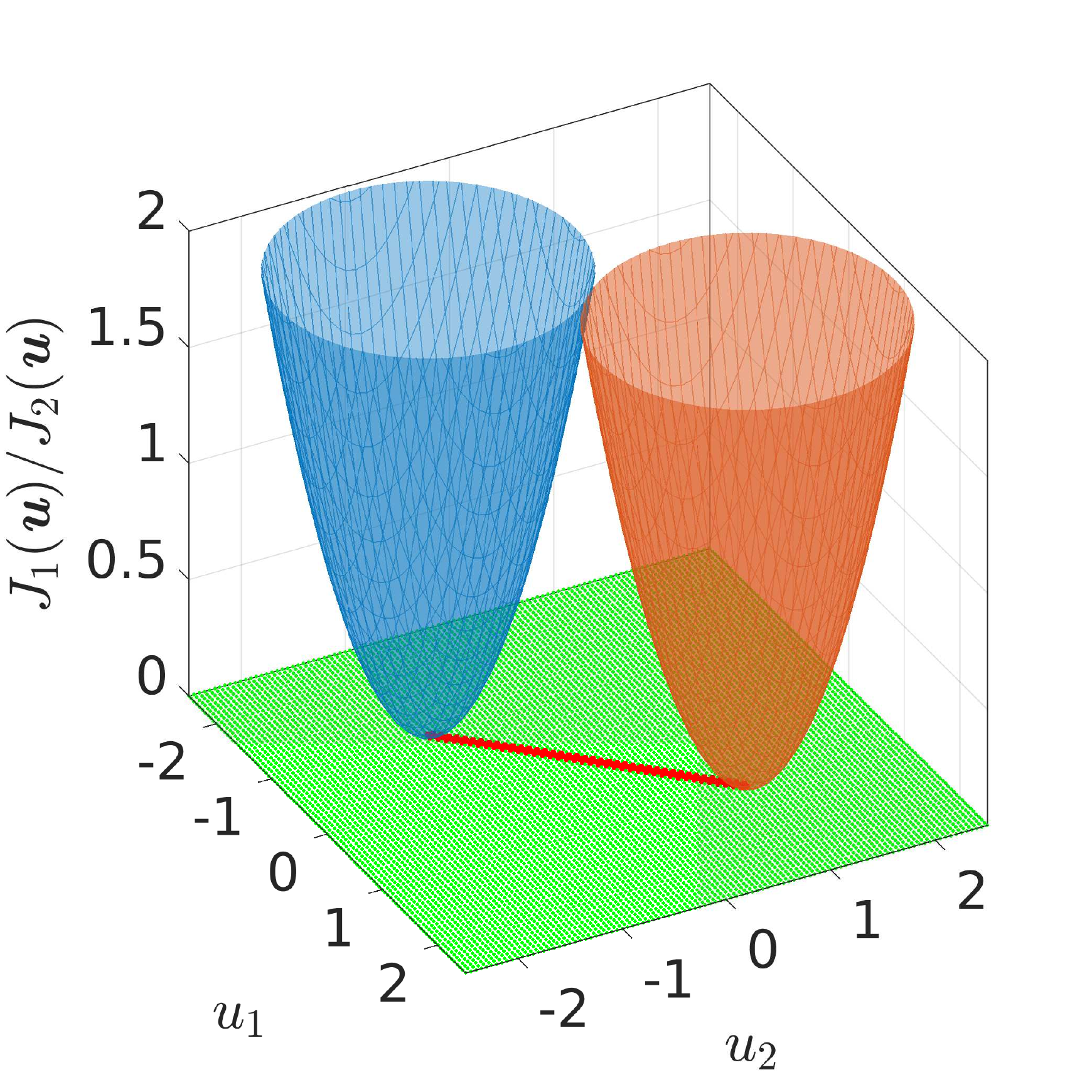} \\ (a)}\hfil
    \parbox[b]{0.24\textwidth}{\centering \includegraphics[width=0.28\textwidth]{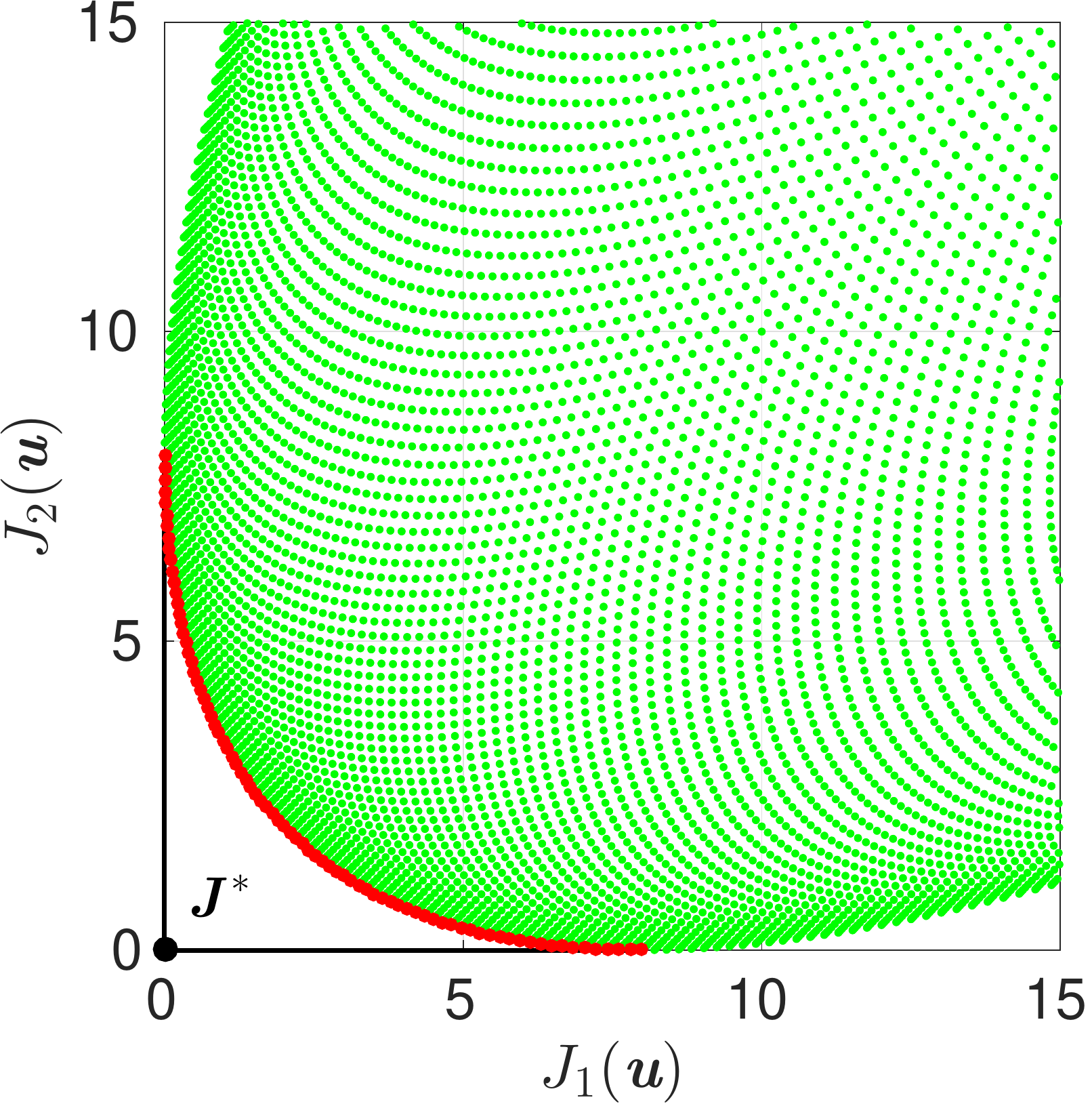} \\ (b)}
    \caption{Red lines: Pareto set $\PS$ (a) and Pareto front $\PF$ (b) of an example problem (two paraboloids) of the form $\min_{u\in\R^2} J(u)$, $J: \R^2\rightarrow\R^2$. The point $J^* = (0, 0)^\top$ is called the utopian point.}
    \label{fig:MOP_Example}
\end{figure}

There exist many fundamentally different approaches to solve MOCPs. Single-objective optimal control problems can either be solved using a \emph{direct solution method} and discretization, cf., e.g., \cite{OJM11}. Alternatively, indirect methods via the \emph{Pontryagin Maximum Principle} can be applied (for an overview of different methods, see e.g.~\cite{Lib12, Betts98}). Solution methods for MOCPs typically rely on a direct approach by which the optimal control problem is transformed into a finite-dimensional multiobjective optimization problem \cite{ORZG12,LSKV10,SWO+13}. 
Well-established methods for such problems are \emph{scalarization techniques}, \emph{continuation methods} (which make use of the fact that under certain conditions, the Pareto set is a smooth manifold of dimension $k-1$ \cite{Hil01} that can be approximated using \emph{predictor-corrector} schemes), \emph{evolutionary algorithms} \cite{CLV07}, or \emph{set-oriented methods} \cite{DSH05,SWO+13}. In scalarization, ideas from single objective optimization theory are extended to the multiobjective situation by transforming the MOCP into a sequence of scalar-valued problems such that the Pareto set is approximated by a finite set of Pareto optimal curves. There exists a large variety of scalarization approaches such as the weighted-sum method, the $\epsilon$-constraint method, normal boundary intersection, or reference point methods \cite{Ehr05}. Since we will use the latter to solve MOCPs, a sketch is given in Figure~\ref{fig:MOP_RP}. In the reference point method, Pareto optimal solutions are obtained by minimizing the euclidean distance between a feasible point $J(u^{(i)})$ and an infeasible \emph{target} $T^{(i)} < J(u^{(i)})$:
\begin{equation}\label{eq:ReferencePoint} \tag{RP}
	\begin{aligned}
		&\min_{u^{(i)} \in \SetU} \left\| \hat{J}(x_0,u^{(i)}) - T^{(i)} \right\|_2^2
		\\
		\hat{g}_i(x_0,\mathfrak{u}^{(i)}) &\leq 0, \quad i=1,\ldots,l, \qquad ~ ~ t \in (t_0,t_e],  \\
		\hat{h}_j(x_0,\mathfrak{u}^{(i)}) &= 0, \quad j=1,\ldots,m, \qquad t \in (t_0,t_e], 
	\end{aligned}
\end{equation}
Once two points on the Pareto front are known, these can be used to approximate the tangent space of the front and construct the target  $T^{(i+1)}$ of the next problem, cf.~Figure~\ref{fig:MOP_RP}~(b). This is realized by shifting the point first parallel (i.e., along the vector $\hat{J}(x_0,u^{(i)})-\hat{J}(x_0,u^{(i-1)})$) and then orthogonal (parallel to the direction $T^{(i)}-\hat{J}(x_0,u^{(i)})$) to the Pareto front. In order to accelerate the solution for the next scalar problem \eqref{eq:ReferencePoint}, a predictor $u^{(p,i+1)}$ is computed by linear extrapolation from the points $u^{(i)}$ and $u^{(i-1)}$.
This way, an almost equidistant covering of the front can be obtained. For a more detailed description see, e.g., \cite[pp.~24--26]{Pei17}.
\begin{figure}[h!]
	\centering
	\parbox[b]{0.24\textwidth}{\centering \includegraphics[width=0.3\textwidth]{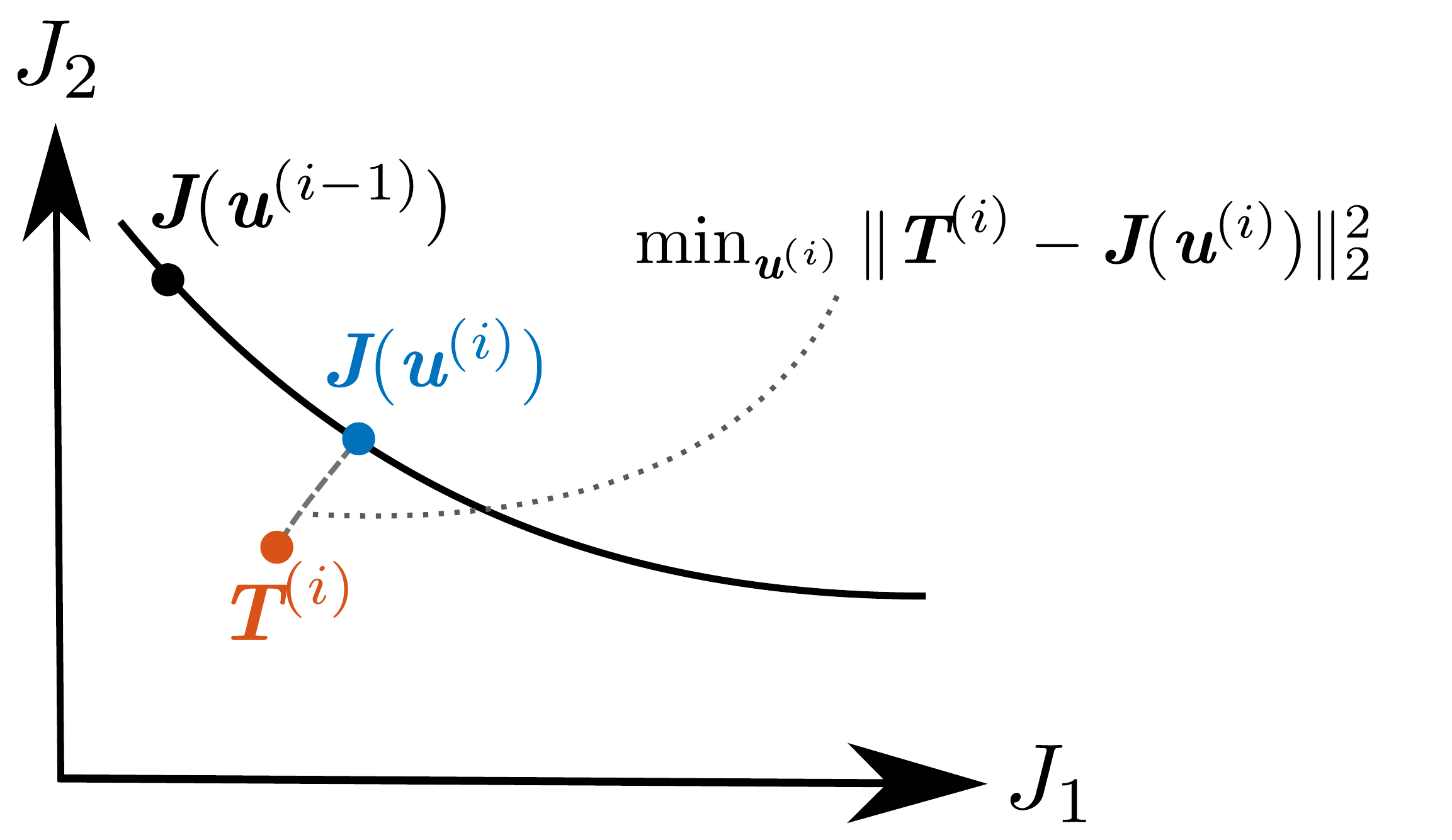} \\ (a)}\hfil
	\parbox[b]{0.24\textwidth}{\centering \includegraphics[width=0.28\textwidth]{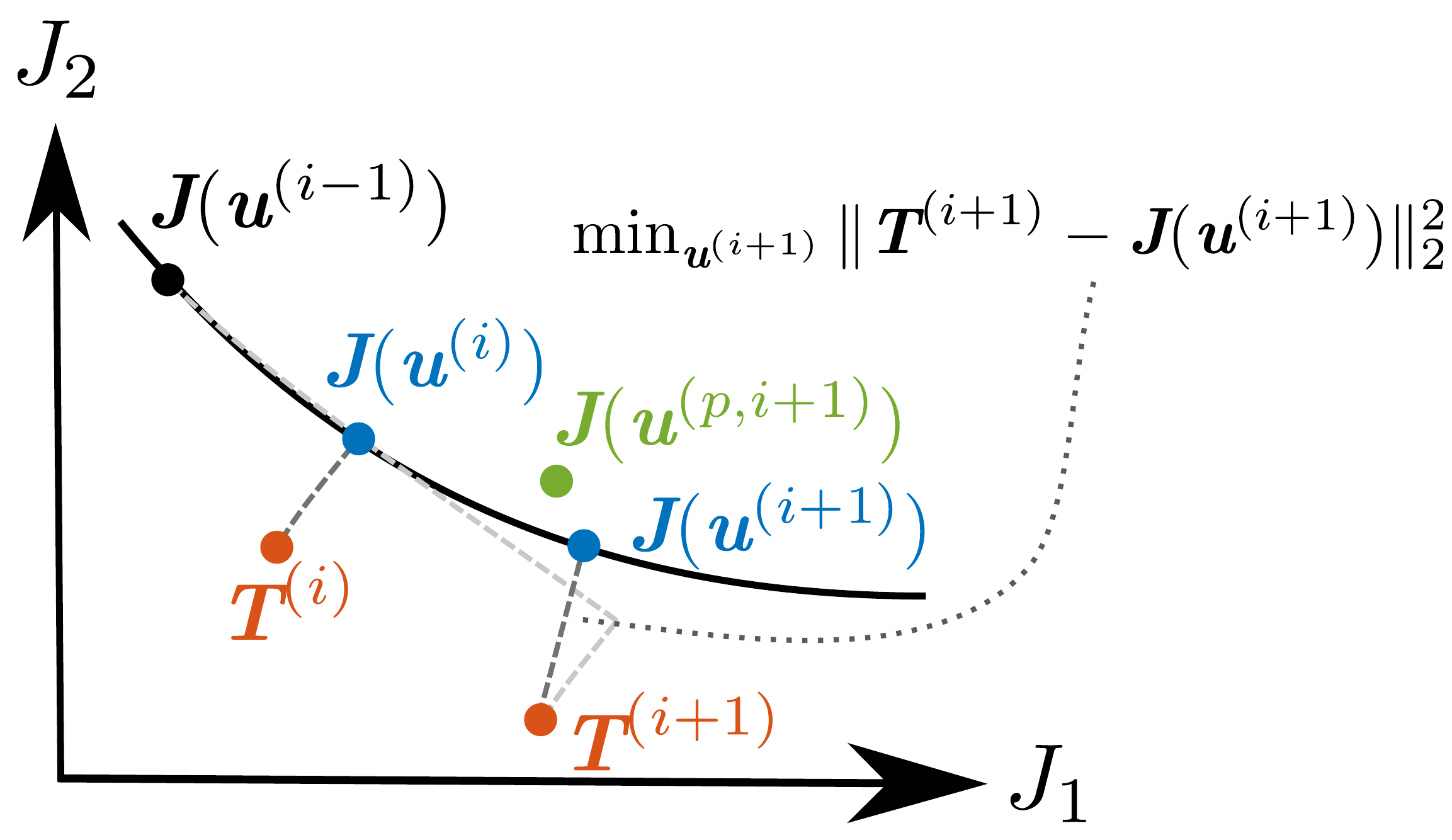} \\ (b)}
	\caption{Reference point method. (a) Solution of the $i^{\mathsf{th}}$ scalar problem. (b) Construction of the next target $T^{(i+1)}$, prediction step $u^{(p,i+1)}$, and solution of the next scalar control problem.}
	\label{fig:MOP_RP}
\end{figure}

\subsection{Model predictive control}
\label{subsec:MOMPC_MPC}
The solution of \eqref{eq:MOCP} provides an open loop control for the system under consideration. However, for real systems it is often insufficient to determine a control input a priori due to unforeseen events, disturbances, or model inaccuracies. A remedy to this issue is MPC \cite{GP17}, where open-loop problems are solved repeatedly on finite horizons (cf.~Figure~\ref{fig:MPC}). Using a model of the system dynamics, an open-loop optimal control computed over the \emph{prediction horizon} of length $ph$, where $p \in \N^{>0}$ and $h \in \R^{>0}$ is the sample time or \emph{control horizon}. This implies that \eqref{eq:MOCP} is solved with
\begin{equation*}
	\begin{aligned}
		t_0 &= t_s, \\ t_e &= t_s + hp = t_{s+p}, \\ x_0 &= x(t_s),
	\end{aligned}
\end{equation*}
for a moving horizon $s=0,1,\ldots$. The first part of the solution is then applied to the plant while the optimization is repeated with the prediction horizon moving forward by one sample time $h$. For this reason, MPC is also referred to as \emph{moving horizon control} or \emph{receding horizon control}.

\begin{figure}[t]
	\centering
	\includegraphics[width=.35\textwidth]{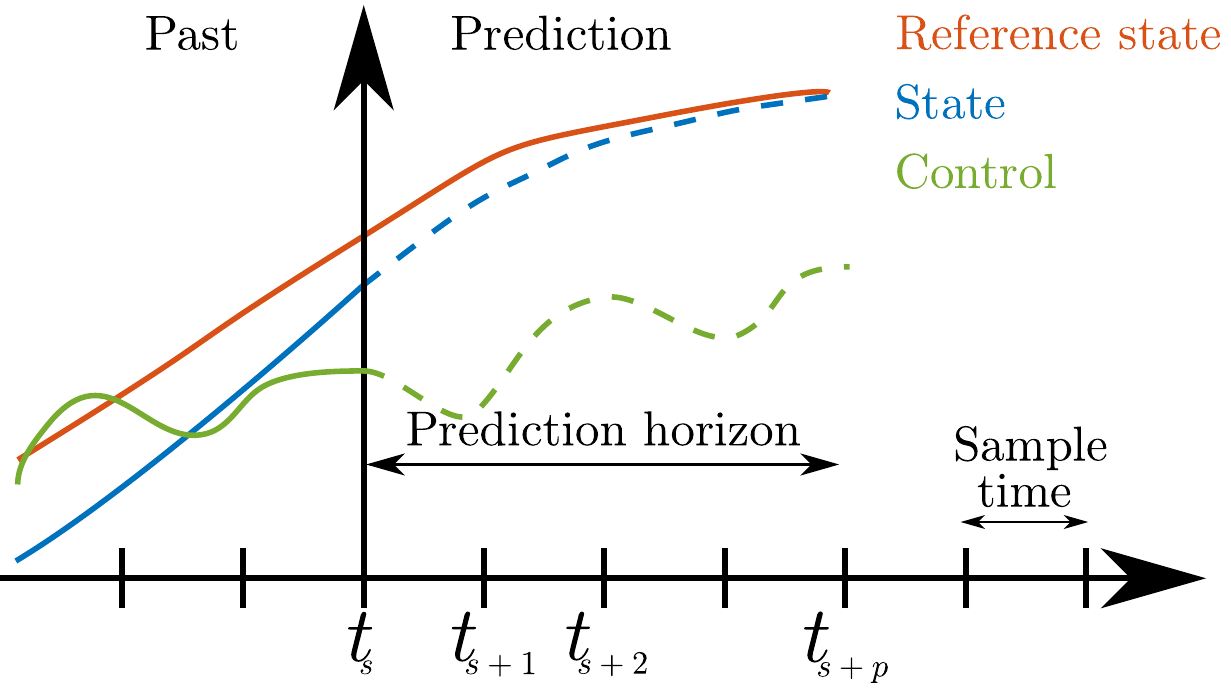}
	\caption{Sketch of the MPC method. Due to the real-time constraints, the optimization problem has to be solved faster than the sample time $h$.}
	\label{fig:MPC}
\end{figure}

Well known extensions of MPC which we will utilize in this article are \emph{economic MPC} \cite{RA09,DAR11}, where the goal is to increase economic performance instead of stabilizing the system, and \emph{explicit MPC} \cite{BMDP02}, see also \cite{AB09} for a survey. In explicit MPC, the MPC problem is reformulated as a multiparametric optimization problem which can be solved in an offline phase and the solutions are stored in a library. During the MPC loop, the computation of an optimal solution is then replaced by extracting the optimal input from the library.

All MPC approaches have in common that they yield a closed-loop behavior. On the downside, we have to solve the optimal control problem within the sample time $h$. This can be in the order of seconds or minutes (in the case of chemical processes) down to a few microseconds, for example in power electronics applications. Since this is already challenging for scalar-valued MPC problems, considering multiple objectives is clearly infeasible without taking further measures. These measures can be, e.g., to apply scalarization using weighted sums \cite{BP09a}, the $\epsilon$-constraint method \cite{Zav15}, reference point methods \cite{ZFT12}, or game-theoretic approaches \cite{MMC11}. Alternatively, one can compute a crude approximation of the entire front \cite{LBK08,GGG+12,NCS+14} or compute Pareto optimal controller parameters offline \cite{KWTD11,HNS+13}. An extensive survey of feedback control with multiple objectives can be found in \cite{PD18}. Before describing our method for addressing the real-time requirements in Section~\ref{sec:Algorithm}, we will first discuss the importance of symmetries in the next section since these will play a crucial role for the cost of the offline phase.

\section{Symmetries in dynamical systems and MPC}
\label{sec:Symmetries}

Symmetries in dynamical systems and optimal control problems can be used to reduce the complexity of the underlying problems and thus making numerical computations faster and more efficient. In this work, we consider continuous symmetries that can be described by a Lie group action. For dynamical systems, this means that translations or rotations of a trajectory lead to another trajectory of the control system. This is a very useful property in control, because a solution trajectory that has been designed for one specific situation can be used for another situation as well. For instance, the same turn maneuver for a helicopter can be applied in many situations since it does not explicitly depend on the absolute position in space. 

While in general symmetry in control is a well-established concept (see, e.g., \cite{BuLe04, Bloch}), it was first exploited by Frazzoli et al.~\cite{Frazzoli2001,FrDaFe05} for the design of efficient motion planning methods which take the dynamics of the control system into account. Following the idea of quantization (see also \cite{FrBu02,Kob08,FOK10}), such so-called \emph{motion primitives} are generated by solving optimal control problems for intermediate problems which can be combined into various sequences. The problem is thus reduced to searching for the optimal sequence out of all admissible sequences in a library of motion primitives which can be realized using global search methods. Extensions towards hybrid systems have been considered in \cite{FO12-ECMI,FHO15}.

In this work, rather than using motion planning approaches, we exploit the motion primitive concept to design explicit MPC algorithms for nonlinear MOCPs. This means that by identifying symmetries in the MOCP, Pareto sets are valid in multiple situations. Consequently, we only have to compute one representative which significantly reduces the computational effort. Before adapting the concept to symmetries in model predictive control problems, we first introduce symmetries in dynamical control systems.

\subsection{Symmetries in dynamical control systems}
We formally describe symmetries by a finite-dimensional Lie group $G$ and its group action $\psi: \R^{n_x} \times G \rightarrow \R^{n_x}$. For each $g\in G$, we denote by $\psi_g: \R^{n_x} \rightarrow \R^{n_x}$ the diffeomorphism defined by $\psi_{g}: = \psi(\cdot,g)$.  
A dynamical control system described by 
\begin{equation}
	\dot{x}(t)  = f(x(t),u(t))\quad \forall t \in [t_0,t_e],\quad x(t_0)=x_0, \label{eq:diff}
\end{equation}
is \textit{invariant under the group action $\psi$}, or equivalently, \textit{$G$ is a symmetry group for the system \eqref{eq:diff}}, if for all $g\in G$, $x_0\in \R^{n_x}$, $t\in [t_0,t_e]$ and $u \in \mathcal{U}$ it holds
\begin{equation}
	\psi_g(\varphi_u(x_0,t)) = \varphi_u(\psi_g(x_0),t). \label{eq:Invariance}
\end{equation}
This means that the group action on the state commutes with the flow. Note that the invariance under a group action implies equivalence of trajectories in the following sense.

\begin{definition}[Equivalence of trajectories]
	Two trajectories $\pi_{1}: t \in [{t_{0,1},t_{e,1}}] \mapsto (x_{1}(t), u_{1}(t))$ and
	$\pi_{2}: t \in [{t_{0,2},t_{e,2}}] \mapsto (x_{2}(t), u_{2}(t))$ of equation \eqref{eq:diff} are \textit{equivalent}, if it holds that 
	\begin{itemize}
		\item[(i)] {$t_{e,1} - t_{0,1} = t_{e,2} - t_{0,2}$} and 
		\item[(ii)] there exist an $g \in {G}$ and an $T \in \R$ such that $x_{1}(t) = \psi_{g}(x_{2}(t-T))$ and $u_{1}(t)  = u_{2}(t-\text{T}) \, \forall \, t \in [{t_{0,1},t_{e,1}}] $.
	\end{itemize}
	\label{equiv}
\end{definition}	

This means that two trajectories are equivalent if they can be exactly superimposed through time translation and the action of the symmetry group. 
All equivalent trajectories can be summed up in an equivalence class. By a slight abuse of notation, we call the equivalence class, but also its representative a \emph{motion primitive} (cf.\ \cite{FrDaFe05}). 
Thus, only one representative has to be stored in a motion library and can then be used in different regions of the state space through transformation by the symmetry action.

\begin{remark}\label{remark:equif}
	Symmetry of a dynamical control system can also be described by the equivariance of the underlying vector field, i.e., by the condition
	\begin{equation}\label{eq:equivariance}
		f(\psi_g(x),u) = D_x\psi_g(f(x,u)) \quad \forall  x\in \R^{n_x}, g\in G,
	\end{equation}
	where $D_x\psi_g: \R^{n_x} \rightarrow \R^{n_x}$ is the tangent lift of $\psi_g$ which acts on $v=\dot{x}$ as $D_x\psi_g(v) = \frac{d}{dx}\psi_g(x) \cdot v$. Is is easy to see that the equivariance condition \eqref{eq:equivariance} is equivalent to the invarince condition \eqref{eq:Invariance}. A direct application of the definition of the flow, equation \eqref{eq:Invariance} and its time derivative shows
	\begin{align*}
		f(\psi_g(x),u) &= f(\psi_g(\varphi_u(x_0,t)),u)\stackrel{\eqref{eq:Invariance}}{=} f (\varphi_u(\psi_g(x_0),t),u) = \frac{d}{dt} \varphi_u (\psi_g(x_0),t) \stackrel{\eqref{eq:Invariance}}{=} \frac{d}{dt} \psi_g(\varphi_u(x_0,t)) \\
		&= \frac{d}{dx} \psi_g(\varphi_u(x_0,t)) \cdot \frac{d}{dt} \varphi_u(x_0,t) = D_x\psi_g( f(x,u)).
	\end{align*}
\end{remark}

Typical symmetries of mechanical systems are translational and rotational symmetries correspondig to Lie groups $G = \R^n$ (translational symmetries), $G=SO(n)$ (rotational symmetries) and $G=SE(n) \approx SO(n) \times \R^n$ (combined rotational and translational symmetries). The corresponding action is then given by $\psi_g(x) = Rx + \Delta x$ with $R\in SO(n)$ and $\Delta x \in \R^n$. Here, $SO(n)$ is the special orthogonal group, which can be represented by the set of matrices $\{ R\in\R^{n,n} \, | \, R^{\top} R=I, \text{det}(R)=1  \}$. The dimension of a Lie group is given by the number of elements required to represent a Lie group element $g\in G$. For the examples above, we have $\text{dim}(\R^n) = n$, $\text{dim}(SO(n)) = n(n-1)/2$ and $\text{dim}(SE(n)) = n(n-1)/2 +n$.

\subsection{Symmetries in MPC problems}

In the following, we want to take advantage of symmetries in optimal control problems. More precisely, we want to identify Pareto optimal solutions of \eqref{eq:MOCP} that remain Pareto optimal when the initial conditions are transformed by the symmetry group action such that 
\begin{align}
	\arg \min_{u} \hat{J}(x_0,u) = \arg \min_{u} \hat{J}(\psi_g(x_0),u) \quad \forall g \in G. \label{eq:Invariance_MOCP0}
\end{align}
Thus, we require the Pareto set to be invariant under group actions on the initial conditions. Symmetries in single-objective, linear-quadratic explicit MPC have been studied in \cite{DB12}, the relation to our approach will be discussed in Example~\ref{ex:LQ}.

The following theorem provides conditions under which equation \eqref{eq:Invariance_MOCP0} holds. In principle, it states that each representative of a motion primitive of the dynamical control system (Condition 1) has to provide the same cost (up to linear transformations) (Condition 2), and if a trajectory $(x(t),u(t))$ satisfies the constraints, then every representative of the same equivalence class has to satisfy the constraints as well (Condition 3).

\begin{theorem}[Symmetry of \eqref{eq:MOCP}]\label{prop:symMOCP}
	Let $\SetX = W^{1,\infty}([t_0,t_e], \R^{n_x})$ and $\SetU = L^{\infty}([t_0,t_e], \R^{n_u})$. If
	\begin{enumerate}
		\item the dynamics are invariant under the Lie Group action $\psi$, i.e.~equation~\eqref{eq:Invariance} holds for $t\in [t_0,t_e]$;
		\item there exist $\alpha,\beta,\delta \in \R$, $\alpha\not=0$, such that the cost functions $C_i$ and the Mayer terms $\Phi_i$, $i=1,\ldots,k$, are invariant under the Lie Group action $\psi$ up to linear transformations, i.e.,
		\begin{equation}\label{eq:symmC}
			C_i(\psi_g(x),u) = \alpha C_i(x,u)+\beta
		\end{equation}
		and
		\begin{equation}\label{eq:symmPhi}
			\Phi_i(\psi_g(x_e)) = \alpha\Phi_i(x_e)+\delta \qquad \mbox{for}~i=1,\ldots,k;
		\end{equation}
		\item the constraints ${g}_i$, $i=1,\ldots,l$ and ${h_j}$, $j=1,\ldots,m$, are invariant under the Lie Group action $\psi$, i.e.,
		\begin{align}
			{g}_i(\psi_g(x),u) &= {g}_i(x,u) \qquad \,\mbox{for}~i=1,\ldots,l, \label{eq:symmg} \\
			{h}_j(\psi_g(x),u) &= {h}_j(x,u) \qquad \mbox{for}~j=1,\ldots,m,\label{eq:symmh}
		\end{align}
	\end{enumerate}
	then we have
	\begin{align}
		\arg \min_{u} \hat{J}(\psi_g(x_0),u) = \arg \min_{u} \hat{J}(x_0,u) \quad \forall g \in G. \tag{\ref{eq:Invariance_MOCP0}} \label{eq:Invariance_MOCP}
	\end{align}
	We say that problem~\eqref{eq:MOCP} is \textit{invariant under the Lie group action $\psi_g$}, or equivalently, \textit{$G$ is a symmetry group} for problem~\eqref{eq:MOCP}. 
	
\end{theorem}

\begin{proof}
	\underline{Feasibility:} Let $u$ be a feasible curve of problem \eqref{eq:MOCP} and let $\varphi_u(x_0,t)$ be the solution of the initial value problem \eqref{eq:diff}. We now consider problem \eqref{eq:MOCP} with initial value $\psi_g(x_0)$, i.e.~the initial value $x_0$ is tansformed by the symmetry group action. 
	Substituting $u$ into the equality and inequality constraints of the transformed \eqref{eq:MOCP} yields (for $i=1,\ldots,l$ and $j=1,\ldots,m$):
	\begin{equation*}
	\begin{split}
		\hat{g}_i(\psi_g(x_0),\mathfrak{u}) &={g}_i(\varphi_u(\psi_g(x_0),t),u(t))  \\
		&\stackrel{\eqref{eq:Invariance}}{=}  {g}_i(\psi_g(\varphi_u(x_0,t)),u(t)) \\
		&\stackrel{\eqref{eq:symmg}}{=}  {g}_i(\varphi_u(x_0,t),u(t))\\
		& = \hat{g}_i(x_0,\mathfrak{u})\le 0
	\end{split} 
	\qquad\qquad \mbox{and} \qquad\qquad
	\begin{split}
		\hat{h}_j(\psi_g(x_0),\mathfrak{u}) &={h}_j(\varphi_u(\psi_g(x_0),t),u(t))  \\
		&\stackrel{\eqref{eq:Invariance}}{=}   {h}_j(\psi_g(\varphi_u(x_0,t)),u(t))\\
		&\stackrel{\eqref{eq:symmh}}{=}   {h}_j(\varphi_u(x_0,t),u(t)) \\
		& =\hat{h}_j(x_0,\mathfrak{u}) = 0.
	\end{split}
	\end{equation*}
	Thus, $u$ is also a feasible curve for problem \eqref{eq:MOCP} with initial value $\psi_g(x_0)$.
	
	\underline{Optimality:}
	Let $u \in \arg \min_{u} \hat{J}(x_0,u)$ and assume there exists an $\tilde{u}$ such that
	\begin{align*}
		&\hat{J}(\psi_g(x_0),\tilde{u}) < \hat{J}(\psi_g(x_0),u)\\
		\Leftrightarrow \quad & {J}(\varphi_{\tilde{u}}(\psi_g(x_0),\cdot),\tilde{u}) < {J}(\varphi_u(\psi_g(x_0),\cdot),{u})\\
		\stackrel{\eqref{eq:Invariance}}{\Leftrightarrow} \quad & \int_{t_0}^{t_e}C_i(\psi_g(\varphi_{\tilde{u}}(x_0,t)),\tilde{u}(t))\, dt + \Phi_i(\psi_g(\varphi_{\tilde{u}}(x_0,t_e))) \\
		&< \int_{t_0}^{t_e}C_i(\psi_g(\varphi_u(x_0,t)),u(t))\, dt + \Phi_i(\psi_g(\varphi_u(x_0,t_e))) \;\forall i\\
		\stackrel{\eqref{eq:symmC},\eqref{eq:symmPhi}}{\Leftrightarrow} \quad & \alpha\left( \int_{t_0}^{t_e}C_i(\varphi_{\tilde{u}}(x_0,t),\tilde{u}(t))\, dt +\Phi_i(\varphi_{\tilde{u}}(x_0,t_e))  \right)+\varepsilon \\
		&< \alpha\left( \int_{t_0}^{t_e}C_i(\varphi_u(x_0,t),{u}(t))\, dt +\Phi_i(\varphi_u(x_0,t_e))  \right)+\varepsilon \;\forall i\\
		\Leftrightarrow \quad &   {J}(\varphi_{\tilde{u}}(x_0,\cdot),\tilde{u}) < {J}(\varphi_u(x_0,\cdot),{u})\\
		\Leftrightarrow \quad & \hat{J}(x_0,\tilde{u}) < \hat{J}(x_0,u)
	\end{align*}
	with $\varepsilon = \beta+\delta$. This is a contradiction to $u \in \arg \min_{u} \hat{J}(x_0,u)$ and consequently $u \in \arg \min_{u}\hat{J}(\psi_g(x_0),{u})$. Following the above steps backwards yields 
	\[u\in \arg \min_{u}\hat{J}(\psi_g(x_0),{u})\Rightarrow u\in \arg \min_{u}\hat{J}(x_0,{u}).\] 
	Thus, the Pareto sets for problems \eqref{eq:MOCP} with initial values $x_0$ and $\psi_g(x_0)$ are identical, i.e., the Pareto set is invariant under group actions on initial conditions.
\end{proof}

Theorem~\ref{prop:symMOCP} states that if the objective function and the constraints are also invariant (up to linear transformations) under the same group action as the dynamical control system, then all trajectories contained in an equivalence class defined by \eqref{eq:Invariance} will also be contained in an equivalence class defined by \eqref{eq:Invariance_MOCP}. However, this class may contain more solutions since we do not explicitly pose restrictions on the state but only require the solutions of \eqref{eq:MOCP} to be identical. This leads to the following corollary.

\begin{corollary}\label{cor:InvPS}
	Let $\SetX = W^{1,\infty}([t_0,t_e], \R^{n_x})$ and $\SetU = L^{\infty}([t_0,t_e], \R^{n_u})$. If
	\begin{enumerate}
		\item there exist $\alpha,\beta,\delta \in \R$, $\alpha\not=0$, such that the cost functions $\hat{C}_i$ and the Mayer terms $\hat{\Phi}_i$, $i=1,\ldots,k$, are invariant under the Lie Group action $\psi$ up to linear transformations, i.e.,
		\begin{equation}\label{eq:symmChat}
			\hat C_i(\psi_g(x_0),\mathfrak{u}) = \alpha \hat C_i(x_0,\mathfrak{u})+\beta
		\end{equation}
		and
		\begin{equation}\label{eq:symmPhihat}
			\hat\Phi_i(\psi_g(x_0),\mathfrak{u}) = \alpha\hat\Phi_i(x_0,\mathfrak{u})+\delta \qquad \mbox{for}~i=1,\ldots,k;
		\end{equation}
		\item the constraints $\hat{g}_i$, $i=1,\ldots,l$ and $\hat{h_j}$, $j=1,\ldots,m$, are invariant under the Lie Group action $\psi$, i.e.
		\begin{align}
			\hat{g}_i(\psi_g(x_0),\mathfrak{u}) &= \hat{g}_i(x_0,\mathfrak{u}) \qquad \,\mbox{for}~i=1,\ldots,l,\label{eq:symmghat} \\
			\hat{h}_j(\psi_g(x_0),\mathfrak{u}) &= \hat{h}_j(x_0,\mathfrak{u}) \qquad \mbox{for}~j=1,\ldots,m,\label{eq:symmhhat}
		\end{align}
	\end{enumerate}
	then \eqref{eq:Invariance_MOCP} holds.
\end{corollary}

\begin{proof}
	The proof follows along the lines of the proof of Theorem~\ref{prop:symMOCP}, where the invariance under group actions can directly applied to $\hat{g}_i, i=1,\ldots,l$, $\hat{h}_j, j=1,\ldots,m$, $\hat{C}_i$ and $\hat{\Phi}_i, i=1,\ldots,k$, without using the invariance of the dynamical control system.
\end{proof}

\begin{remark}[(Group actions on controls and parameters)] \label{rem:Parameters}
	Rather than considering transformations of state trajectories by group actions $\psi_g$ with $g\in G$ only, also Lie group actions $\chi_h$ on the control trajectories $u$ as well as Lie group actions $\xi_l$ on the parameters $\gamma\in \R^{n_\gamma}$ can be taken into account with $h$ and $l$ being elements of the Lie groups $H$ and $L$, respectively. Thus, we obtain the group action triple $(\psi_g,\chi_h, \xi_l): \R^{n_x} \times \R^{n_u} \times \R^{n_\gamma}\rightarrow \R^{n_x} \times \R^{n_u} \times \R^{n_\gamma}$. For a parameter-dependent flow $\varphi_u(x_0,t;\gamma)$ the invariance condition \eqref{eq:Invariance} for the dynamical control system is then replaced by
	\begin{equation}
		\psi_g(\varphi_u(x_0,t; \gamma)) = \varphi_{\chi_h u}(\psi_g(x_0),t; \xi_l(\gamma)), \label{eq:Invarianceu}
	\end{equation}
	meaning that for two trajectories being equivalent, we also allow for a transformation of the control $u$ and the parameter $\gamma$ by the Lie group actions $\chi_h$ and $\xi_h$, respectively. For a parameter-dependent cost function, the corresponding invariance condition for problem \eqref{eq:MOCP} then reads
	\begin{align}
		\arg \min_{u} \hat{J}(x_0,u,\gamma) = \arg \min_{u} \hat{J}(\psi_g(x_0),\chi_h(u), \xi_l(\gamma)) \quad \forall g \in G, h\in H, l \in L. \label{eq:Invariance_MOCPu}
	\end{align}
	By Theorem~\ref{prop:symMOCP}, \eqref{eq:Invariance_MOCPu} is satisfied if dynamics, cost functions, Mayer terms and constraints are all invariant under the Lie group action $(\psi_g,\chi_h,\xi_l)$, i.e., in addition to \eqref{eq:Invarianceu}, we have for $\alpha,\beta,\delta \in \R$, $\alpha\not=0$:
	\begin{align}
		C_i(\psi_g(x),\chi_h(u),\xi_l(\gamma))& = \alpha C_i(x,u,\gamma)+\beta\\		
		\Phi_i(\psi_g(x(t_e)),\xi_l(\gamma)) &= \alpha\Phi_i(x(t_e),\gamma)+\delta \quad \text{for $i=1,\ldots,k$}\\
		{g}_i(\psi_g(x),\chi_h(u),\xi_l(\gamma)) &= {g}_i(x,u,\gamma) \quad \text{for $i=1,\ldots,l$}\\
		{h}_j(\psi_g(x),\chi_h(u),\xi_l(\gamma)) &= {h}_j(x,u,\gamma) \quad \text{for $j=1,\ldots,m$}.
	\end{align}
\end{remark}

\begin{example}[Linear-quadratic problems]\label{ex:LQ}
	In the case of single-objective linear-quadratic problems
	and group actions that can be described by matrix multiplications, Theorem~\ref{prop:symMOCP} yields the statements of Proposition~2 and Theorem 4 of \cite{DB12}, which relates symmetries in the dynamics, constraints and cost function to symmetries in linear-quadratic model-predictive control problems and explicit controller functions.
\end{example}

\begin{example}[Parameter-dependent problems]
	A typical example for parameter-dependent problems are tracking problems with cost functions describing the squared distance between the state $x$ and some reference $\gamma$ which has to be tracked, i.e.,
	\[
		C(x,\gamma) = \| x-\gamma\|_2^2.
	\]
	For dynamical control systems being invariant under translations and rotations, i.e., $\psi_g(x) = R\cdot x + \Delta x$, invariance of the cost function is obtained by applying the same Lie group action $\psi_g$ to $\gamma$. This ensures that the distance between state and reference is preserved under the Lie group action, i.e., $C(\psi_g(x),\psi_g(\gamma)) = \| \psi_g(x)-\psi_g(\gamma)\|_2^2 = \| R\cdot x+\Delta x - (R \cdot \gamma + \Delta x)\|_2^2 = C(x,\gamma)$
	where the last equality follows from the orthogonality of $R$.
\end{example}

\begin{example}[Invariance of the $\arg \min$]
	In order to emphasize the meaning of Corollary \ref{cor:InvPS}, let us consider the parameter-dependent multiobjective optimization problem from \cite[Example 3.12]{Wit12} with $J: \R^2 \times \R \rightarrow \R^2$:
	\begin{equation} \label{eq:Witting}
	\min_{u \in \R^2} J(u,\gamma) = \min_{u \in \R^2} \left(\begin{array}{c}
	\frac{1}{2} \left(\sqrt{1 + (u_1 + u_2)^2} + \sqrt{1 + (u_1 - u_2)^2}) + u_1 - u_2 \right) + \gamma e^{-(u_1 - u_2)^2} \\
	\frac{1}{2} \left(\sqrt{1 + (u_1 + u_2)^2} + \sqrt{1 + (u_1 - u_2)^2}) - u_1 + u_2 \right) + \gamma e^{-(u_1 - u_2)^2}
	\end{array}\right).
	\end{equation}
	The Pareto sets and fronts for varying values of $\gamma$ are shown in Figure~\ref{fig:Witting}. We see that $\PS$ is invariant under translations in $\gamma$ (i.e., the $\arg \min$ of \eqref{eq:Witting} is invariant under translations in $\gamma$).
	\begin{figure}[h!]
		\centering
		\parbox[b]{0.3\textwidth}{\centering \includegraphics[width=0.3\textwidth]{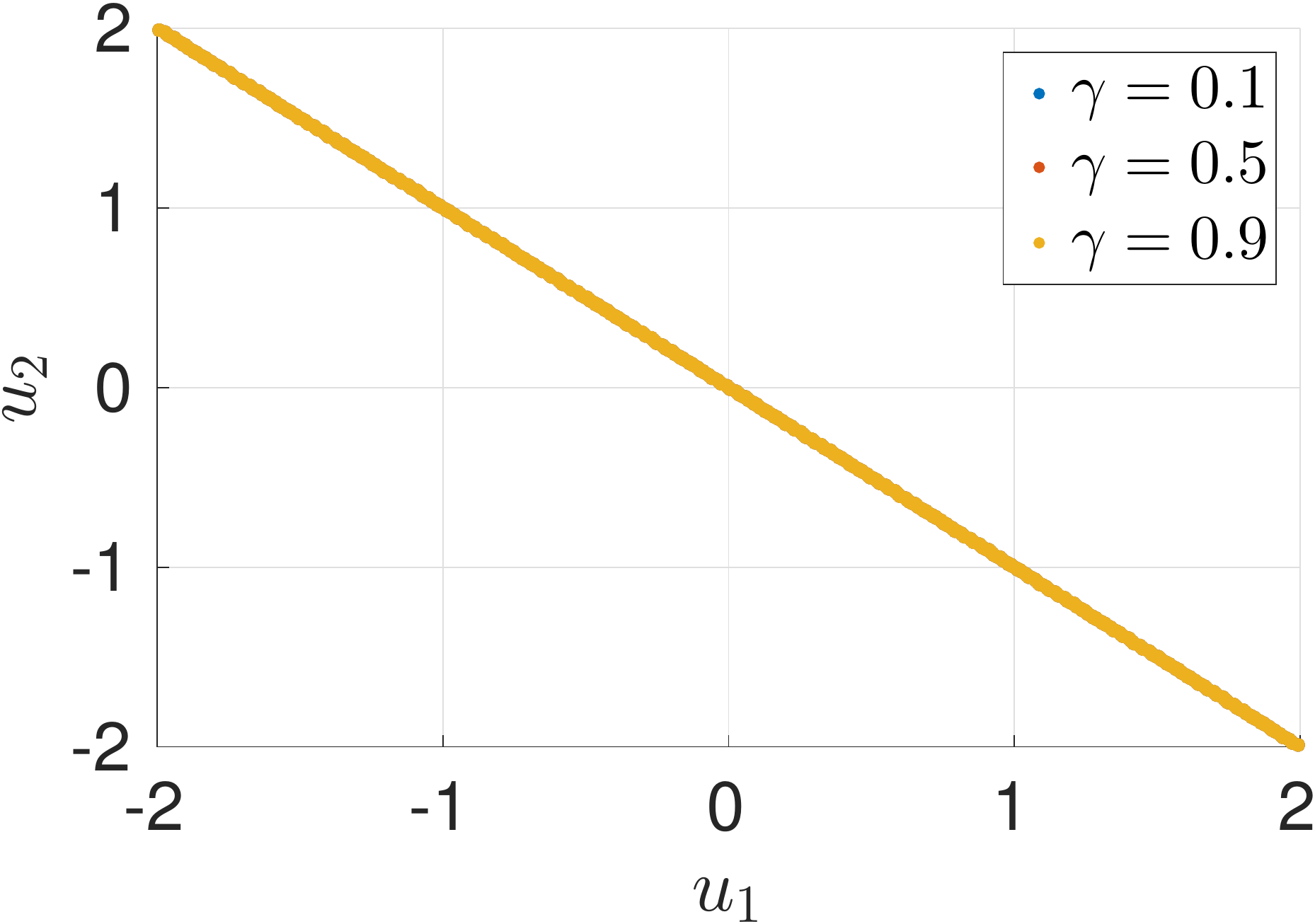} \\ (a)}\hfil
		\parbox[b]{0.3\textwidth}{\centering \includegraphics[width=0.3\textwidth]{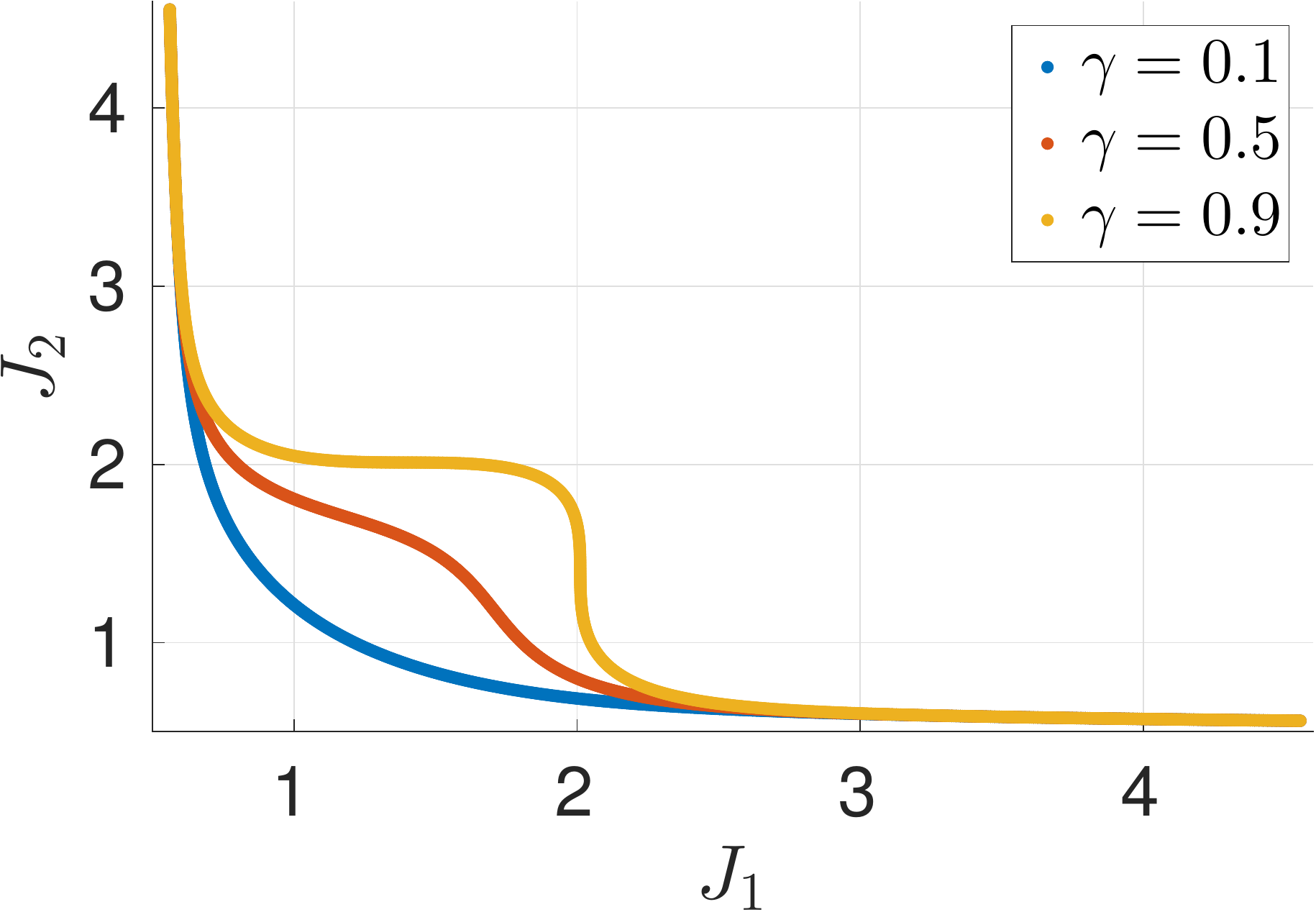} \\ (b)}
		\caption{Pareto set (a) and Pareto front (b) of Problem~\eqref{eq:Witting} for varying values of $\gamma$. Although the fronts vary, $\PS$ is invariant under translations in $\gamma$.}
		\label{fig:Witting}
	\end{figure}
\end{example}

\subsection{Numerical identification of symmetries}
\label{subsec:Sym_Numerical_Identification}

There may be problems where it is very tedious or even impossible (for example when black-box models are involved) to identify symmetries analytically. In this situation, one can use numerical approaches to verify Equation~\eqref{eq:Invariance_MOCPu} (or \eqref{eq:Invariance_MOCP0}, respectively). In the context of multiobjective optimization, this means that the Pareto set is invariant under variations of some parameter $\gamma$:
\[
	\frac{d \PS_{\gamma}}{d \gamma} = 0 \quad \Longleftrightarrow \quad \PS_{\gamma} = \PS_{\overline{\gamma}} \quad \forall ~ \gamma, \overline{\gamma} \in [\gamma_{min}, \gamma_{max}]. 
\]
Numerically, this can be realized by demanding that the Hausdorff $d_h$ distance between Pareto sets at a finite number of $n_{\mathsf{Test}}$ different parameter values is small:
\[ 
	d_h(\PS_{\gamma}, \PS_{\overline{\gamma}})\leq \epsilon \quad \forall ~ \gamma, \overline{\gamma} \in \{\gamma_1, \ldots, \gamma_{n_{\mathsf{Test}}}\}.
\]
In Figure~\ref{fig:Witting}, for instance, the Hausdorff distance between the Pareto sets for different values of $\gamma$ is zero. This concept will be exploited in the second example in Section~\ref{subsec:Examples_EV}.

\section{Explicit multiobjective MPC for nonlinear problems}
\label{sec:Algorithm}
The method that will be used in this article is in the spirit of explicit MPC, i.e., we solve a large number of MOCPs offline and store the corresponding Pareto sets in a library. By exploiting symmetries in the dynamical control system, the number of MOCPs is significantly reduced. In the online phase, we only have to select the correct Pareto set from the library and then choose a compromise solution according to the decision maker's preference. In contrast to the classical motion planning with motion primitives concept, we do not need to store the trajectories but only the optimal controls, i.e., the Pareto sets. This is due to the MPC framework which only requires control values as input to run the plant. Thus, a transformation by the symmetry group action to construct feasible trajectories as in the open loop case is not necessary anymore. First ideas concerning the MPC approach have previously appeared in \cite{PSOB+17}. Before introducing the two phases in detail in the following, we first give a quick introduction to the classical approach (introduced in \cite{BMDP02}) and some extensions.

\subsection{Relation to the single-objective case}
\label{subsec:Relation_LQSQ}
Explicit MPC was introduced by Bemporad et al.~\cite{BMDP02}. In order to avoid prohibitively large costs for solving optimal control problems in real time, they introduced an offline-online decomposition such that during operation, one only has to select the precomputed optimal control from a library. They showed that in the linear-quadratic case, the optimal control problem can be reformulated as a \emph{multi-parametric quadratic programming (mpQP)} problem. In this setting, the optimal control is an affine function of the initial condition $x_0$ such that the state-to-control mapping can be constructed via a finite number of polyhedrons covering the state space. Extensions to the nonlinear case were presented in \cite{Joh02,BF06}. Here, the polyhedrons from the linear-quadratic case are constructed via linear interpolation between nodes at which the solution has been computed, and adaptive procedures are proposed in order to satisfy prescribed error bounds.

For the linear-quadratic case, an extension to multiple objectives could be performed in a straightforward manner using the method of weighted sums, where the vector of objective functions is synthesized into one objective via convex combination using a weight vector $\rho \in [0,1]^k$ (see \cite{Ehr05} for details). Since the resulting MOCP is convex \cite{BP09a}, the weighted sum method is capable of computing all optimal compromises. Fixing the weight $\rho$ yields precisely the setting considered in \cite{BMDP02}. Hence, we can introduce a numerical grid for $\rho$ and solve an mpQP for each value. However, since we want to consider nonlinear problems here, neither the weighted sum method nor the mpQP approach are applicable. Thus, we present an alternative approach in this article. 

For both the linear and the nonlinear case, the cost of the offline phase increases exponentially with the dimension of the parameter. Hence, we will exploit symmetries in the MOCP in order to reduce the parameter dimension and thus, the number of MOCPs that we have to solve in the offline phase. The exploitation of symmetries for single-objective linear-quadratic problems has been studied in \cite{DB12}, and the relation to our approach is discussed in Example~\ref{ex:LQ}.
	
\subsection{Offline phase}
As can be seen in the problem formulation \eqref{eq:MOCP}, we consider a problem which is parametrized with respect to the initial value $x_0 \in \R^{n_x}$. In order to provide the optimal solution to \eqref{eq:MOCP} for each value of $x_0$ (and possibly additional parameters $\gamma \in \R^{n_\gamma}$, cf.~Remark~\ref{rem:Parameters}), we have to solve a large number of problems in the offline phase.
As mentioned previously, the solution is piecewise linear and continuous in the linear-quadratic case and can be computed offline for all parameter values. In the nonlinear case, this structure does no longer exist. Consequently, we would have to solve an infinite number of MOCPs in the offline phase. As a workaround, a natural idea is to discretize $(x_0, \gamma)$ on an equidistant grid and use linear interpolation for intermediate values. The grid size depends critically on the degree of nonlinearity of the system under consideration. Furthermore, the number of MOCPs increases exponentially with the state dimension $n_x$. A reduction of this dimension is therefore highly advisable. In the linear-quadratic case, this has been addressed in \cite{DB12} (cf.~Example~\ref{ex:LQ}).

\begin{algorithm} 
	\caption{Offline phase}
	Given: Lower and upper bounds $x_{0,min}, x_{0,max} \in \R^{n_x}$ and $\gamma_{min}, \gamma_{max} \in \R^{n_\gamma}$, number of grid points $\delta\in \N^{n_x + n_\gamma}$.
	\begin{algorithmic}[1]
		\State Dimension reduction: Decrease dimension of parameter $(x_0,\gamma) \in \R^{n_x + n_\gamma}$ to $(\widetilde{x}_0, \widetilde{\gamma}) \in \R^{\widetilde{n}_x + \widetilde{n}_\gamma}$ by exploiting the symmetry groups $G$ and $L$ (cf.~Section~\ref{sec:Symmetries}).
		\State Construction of library: Create an $(\widetilde{n}_x + \widetilde{n}_\gamma)$-dimensional grid $\mathcal{L}$ for the parameter $(\widetilde{x}_0, \widetilde{\gamma})$ between $(\widetilde{x}_{0,min}, \widetilde{\gamma}_{min})$ and $(\widetilde{x}_{0,max}, \widetilde{\gamma}_{max})$ with $\delta_i$ points in the $i^{\mathsf{th}}$ direction. This results in $N = \prod_{i=1}^{\widetilde{n}_x + \widetilde{n}_{\gamma}} \delta_i$ parameters.
		\State Compute the Pareto sets $\PS_{(\widetilde{x}_0, \widetilde{\gamma})}$ for all $(\widetilde{x}_0, \widetilde{\gamma}) \in \mathcal{L}$.
	\end{algorithmic}
	\label{alg:offline}
\end{algorithm}

The offline phase can now be summarized in Algorithm~\ref{alg:offline}. After identifying symmetries and thereby reducing the dimension of the parameter $(x_0,\gamma) \in \R^{n_x + n_\gamma}$ to $(\widetilde{x}_0, \widetilde{\gamma})\in \R^{\widetilde{n}_x + \widetilde{n}_\gamma}$, where
\[
	\widetilde{n}_x = n_x - \text{dim}(G) \mbox{ and } \widetilde{n}_\gamma = n_\gamma - \text{dim}(L),
\]
we create the library $\mathcal{L}$ as an equidistant multi-dimensional grid and solve an MOCP for each entry of $\mathcal{L}$. The number $N$ of MOCPs that has to be solved can still become very large. However, the solution process can be parallelized very efficiently since the $N$ problems are independent. 

\subsubsection{Automated solution of many MOCPs}
In this section, we describe the numerical procedure to automatically solve MOCPs as is required in step 3 of Algorithm~\ref{alg:offline}. The procedure is summarized in Algorithm~\ref{alg:MOCP_Auto} and visualized in Figure~\ref{fig:RP_Auto}.\footnote{The description here is limited to two objectives. The extension to more objectives can be realized in a straightforward manner. However, it should be noted that in this case, the selection of targets becomes difficult.} 
The procedure consists of solving the scalar problems (i.e., we solve a single-objective optimal control problem for each objective separately), determining target points, solving the resulting scalarized problems \eqref{eq:ReferencePoint}, and cleaning up the Pareto set by removing the \emph{long tails} of the Pareto fronts. These tails are not \emph{properly non-dominated}, i.e., they yield only very small trade-offs ($< \epsilon$), see \cite{Wit12} for details. The last step is done in order to allow for a more meaningful selection process in the online phase.
\begin{algorithm} 
	\caption{Automated solution of \eqref{eq:MOCP} with two objectives}
	Given: Number of targets $n_T \in \N$, parameter $\epsilon$ for proper efficiency, parameter $d_e$.
	\begin{algorithmic}[1]
		\State Solve the scalar problems (i.e., $\min J_1$ and $\min J_2$) individually.
		\State Distribute $n_T$ targets on the ellipse defined by the points $J^*$, $J^*_1 - (d_e, 0)^\top$ and $J^*_2 - (0, d_e)^\top$.
		\State Solve \eqref{eq:ReferencePoint} $n_T$ times for the respective targets.
		\State Remove all points which are not properly Pareto optimal.
	\end{algorithmic}
	\label{alg:MOCP_Auto}
\end{algorithm}
\begin{figure}[h!]
	\centering
	\includegraphics[width=0.25\textwidth]{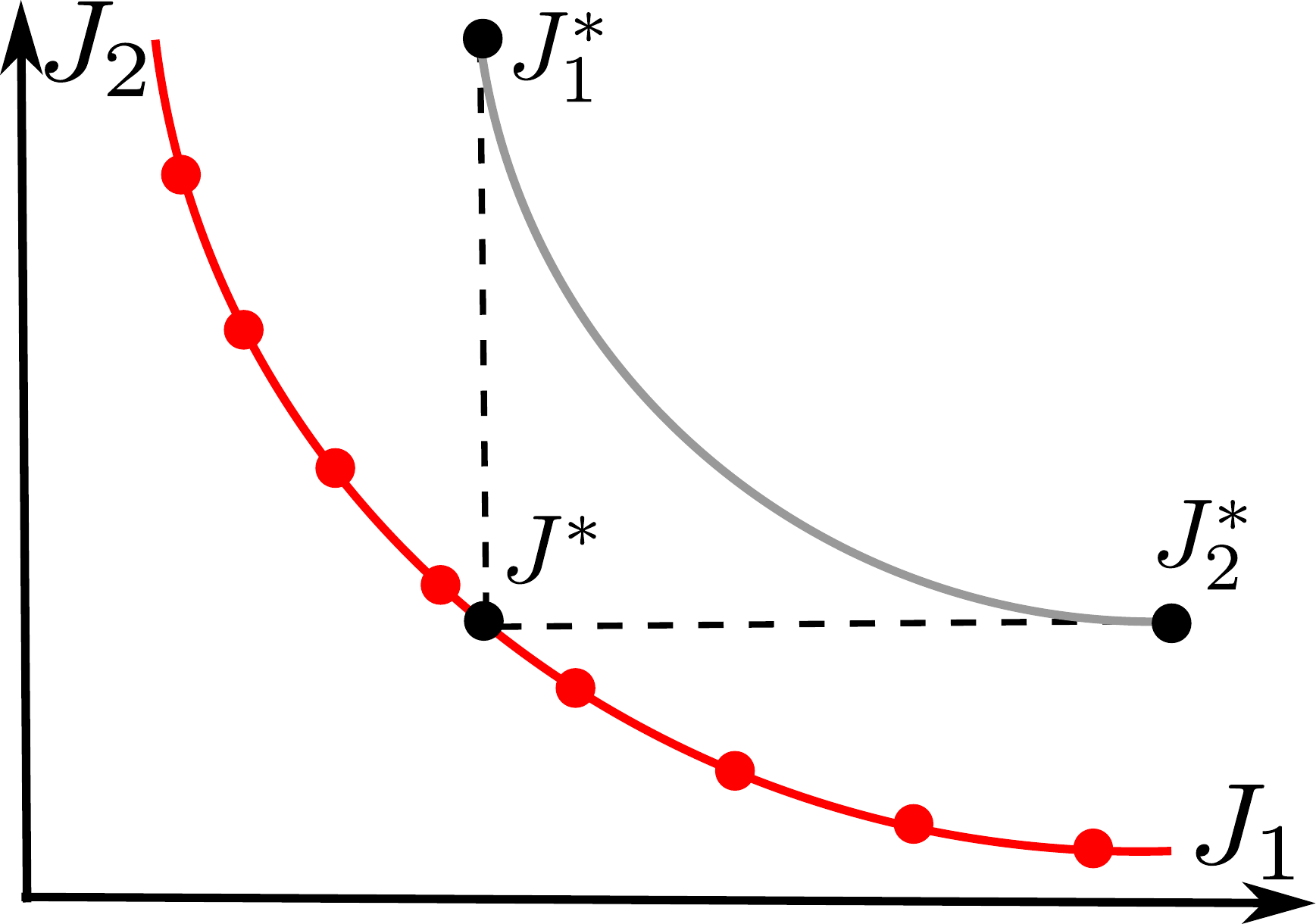}
	\caption{Procedure for automatically solving an MOCP using the reference point method. After determining the two scalar minimizers, a fixed number of reference points is distributed on an ellipse going through the utopian point $J^*$.}
	\label{fig:RP_Auto}
\end{figure}

Since we want to parallelize the computation of the Pareto sets, we solve all problems individually, i.e., without taking any knowledge about prior Pareto sets into account. 
If the solutions have to be computed without parallelization, one can alternatively exploit the fact that under certain smoothness assumptions, the Pareto set depends continuously on the parameter $(\widetilde{x}_0, \widetilde{\gamma})$ such that small variations lead to small variations in $\PS_{(\widetilde{x}_0, \widetilde{\gamma})}$. An approach exploiting this in a set-valued continuation method is presented in \cite[Section 3.3]{Pei17}.

\subsection{Online phase}
The online phase is very similar to a standard MPC approach with the important difference that the MOCP is not solved online but the pre-computed solution is obtained from the library that has been computed in the offline phase. The task is thus to identify the current values for $\widetilde{x}_0$ and $\widetilde{\gamma}$ and then select the corresponding Pareto set $\PS_{(\widetilde{x}_0, \widetilde{\gamma})}$ from the library. According to the decision maker's current preference $\rho \in [0,1]^{k}$ (with $\sum_{i=1}^{k} \rho_i= 1$), a Pareto optimal control is chosen and applied to the plant over the control horizon $t_c$. The online phase is summarized in Algorithm~\ref{alg:online}.

Since we do not know the exact solution for every $(\widetilde{x}_0, \widetilde{\gamma})$, we use linear interpolation between the neighboring entries of $\mathcal{L}$. In order to avoid the interpolation of sets, we first select an optimal compromise and then perform the interpolation. For the single-objective situation, linear interpolation has also been proposed for nonlinear problems, see \cite{Joh02,BF06}.

\begin{algorithm} 
	\caption{Online phase}
	Given: Weight $\rho \in \R^k$ with $\sum_{i=1}^{k} \rho_i = 1$ and $\rho \geq 0$.
	\begin{algorithmic}[1]
		\For{$t = t_0,t_1,t_2,\ldots$}
		\State Obtain the current initial condition $\widetilde{x}_0 = \widetilde{x}(t)$ and the parameter value $\widetilde{\gamma}$ from the plant.
		\State Identify the $2 (\widetilde{n}_x + \widetilde{n}_\gamma)$ neighboring grid points of $(\widetilde{x}_0, \widetilde{\gamma})$ in $\mathcal{L}$ (i.e., closest below and above in each component of $(\widetilde{x}_0, \widetilde{\gamma})$). These points are collected in the index set $\mathcal{I}$.
		\State From each of the corresponding Pareto sets $\PS_{(\widetilde{x}_0, \widetilde{\gamma})_i}$, $i\in\mathcal{I}$, select a Pareto optimal control $u_i$ according to the weight $\rho$.
		\State Compute the distances $d_i$ between the entries of the library and $(\widetilde{x}_0, \widetilde{\gamma})$:
			\begin{align*}
				d_i = \|(\widetilde{x}_0, \widetilde{\gamma})_i - (\widetilde{x}_0, \widetilde{\gamma})\|_2.
			\end{align*}
		\If{$\exists j \in \mathcal{I}$ with $d_j = 0$}
			\State $u = u_j$
		\Else \State
			\begin{align*}
				u = \frac{\sum_{i=1}^{|\mathcal{I}|}\frac{1}{d_i} u_i}{\sum_{i=1}^{|\mathcal{I}|}\frac{1}{d_i}}.
			\end{align*}
		\EndIf
		\State Apply $u$ to the plant for the control horizon length $t_c$. 
		\EndFor
	\end{algorithmic}
	\label{alg:online}
\end{algorithm}

\section{Examples}
\label{sec:Examples}
In this section, the \emph{explicit multiobjective MPC (EMOMPC)} method is validated using two examples from autonomous driving. This problem has raised increasing interest in the past, in particular due to the additional interest in energy efficiency for both ecological reasons and reduced ranges of electric vehicles. A survey on path tracking of autonomous vehicles using motion primitives can be found in \cite{PCY+16}, and several researchers have addressed multiobjective optimal control of vehicles as well, see, e.g., \cite{LSKV10,LLRW11}. 

Here, we will first consider the problem of maneuvering, where we want to stay as close to a reference track as possible while maximizing the driven distance. As a second example, we will revisit the electric vehicle application presented in \cite{PSOB+17}, where the longitudinal dynamics of an electric vehicle have to be controlled in a Pareto optimal manner. Here, we will also address the numerical identification of symmetries discussed in Section~\ref{subsec:Sym_Numerical_Identification}.

\subsection{Multiobjective car maneuvering}
\label{subsec:Example_Car}
As the first example, we consider driving on a race track, where we are interested in optimally determining the steering angle for a vehicle with respect to secure and fast driving. To this end, we consider the well-known \emph{bicycle model} \cite{TL90} (see also \cite{PL14}, where a more complex bicycle model has been used for single objective optimal control of formula one cars). In this model, the dynamics of the vehicle is approximated by representing the two wheels on each axis by one wheel on the centerline (cf.~Figure~\ref{fig:Car_Bicycle} (a)). When assuming a constant longitudinal velocity $v_x$, this leads to a nonlinear system of five coupled ODEs:
\begin{equation} \label{eq:ODE_Bicycle}
	\begin{aligned}
		\dot{x}(t) &= \left(\begin{array}{c}
		\dot{p_1}(t) \\
		\dot{p_2}(t) \\
		\dot{\varTheta}(t) \\
		\dot{v_y}(t) \\
		\dot{r}(t) 
		\end{array}\right) =
		\left(\begin{array}{c}
		v_x(t) \cos(\varTheta(t)) - v_y(t) \sin(\varTheta(t)) \\
		v_x(t) \sin(\varTheta(t)) + v_y(t) \cos(\varTheta(t)) \\
		r \\
		C_1(t) v_y(t) + C_2(t) r(t) + C_3(t) u(t), \\
		C_4(t) v_y(t) + C_5(t) r(t) + C_6(t) u(t)
		\end{array}\right),\ t \in (t_0, t_e], \\
		x(t_0) &= x_0,
	\end{aligned}
\end{equation}
where $x = (p_1,p_2,\varTheta, v_y, r)^\top$ is the state consisting of the position $p = (p_1, p_2)$, the angle $\varTheta$ between the horizontal axis and the longitudinal vehicle axis, the lateral velocity $v_y$ and the yaw rate $r$ (cf.~Figure~\ref{fig:Car_Bicycle} (a)). The vehicle is controlled by the front wheel angle $u$ and the variables
\begin{equation*}
	\begin{aligned}
		C_1(t) &= -\frac{C_{\alpha,f} \cos(u(t)) + C_{\alpha,r}}{m v_x(t)}, \\
		C_3(t) &= \frac{C_{\alpha,f} \cos(u(t))}{m},\\
		C_5(t) &= - \frac{L_f^2 C_{\alpha,f} \cos(u(t)) + L_r^2 C_{\alpha,r}}{I_z v_x(t)},
	\end{aligned} \qquad
	\begin{aligned}
		C_2(t) &= \frac{-L_f C_{\alpha,f} \cos(u(t)) + L_r C_{\alpha,r}}{I_z v_x(t)}, \\
		C_4(t) &= \frac{-L_f C_{\alpha,f} \cos(u(t)) + L_r C_{\alpha,r}}{m v_x(t)} -v_x(t),\\
		C_6(t) &= \frac{L_f C_{\alpha,f} \cos(u(t))}{I_z},
	\end{aligned}
\end{equation*}
have been introduced for abbreviation. The constants therein describe the vehicle's geometry, mass as well as tyre properties, see Table~\ref{tab:Bicycle_Constants}.
\begin{table}[t]
	\caption{Physical constants of the vehicle model.}
	\centering
	\begin{tabular}{ccc}
		\hline
		Variable & Physical property & Numerical value \\ 
		\hline 
		$C_{\alpha,f}$ & Cornering stiffness coefficient (front) & $65100$ \\ 
		$C_{\alpha,r}$ & Cornering stiffness coefficient (rear) & $54100$\\ 
		$L_f$ & Distance front wheel to center of mass & $1$ \\ 
		$L_r$ & Distance rear wheel to center of mass & $1.45$ \\ 
		$m$ & Vehicle mass & $1275$ \\ 
		$I_z$ & Moment of inertia & $1627$ \\ 
		\hline 
	\end{tabular}
	\label{tab:Bicycle_Constants}
\end{table}

\begin{figure}[t!]
	\centering
	\parbox[b]{0.25\textwidth}{\centering \includegraphics[width=0.25\textwidth]{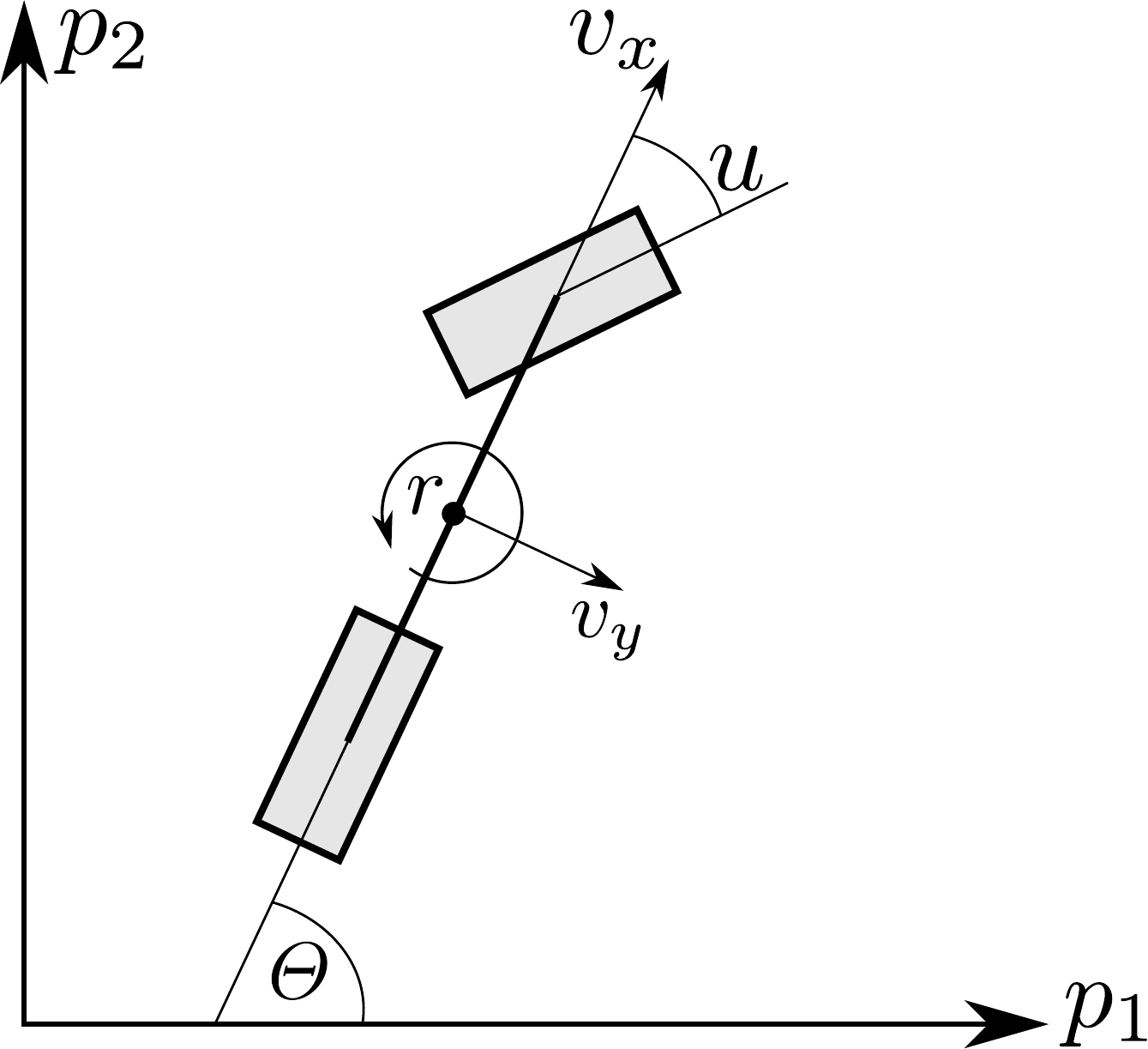} \\ (a)}\hfil
	\parbox[b]{0.42\textwidth}{\centering \includegraphics[width=0.42\textwidth]{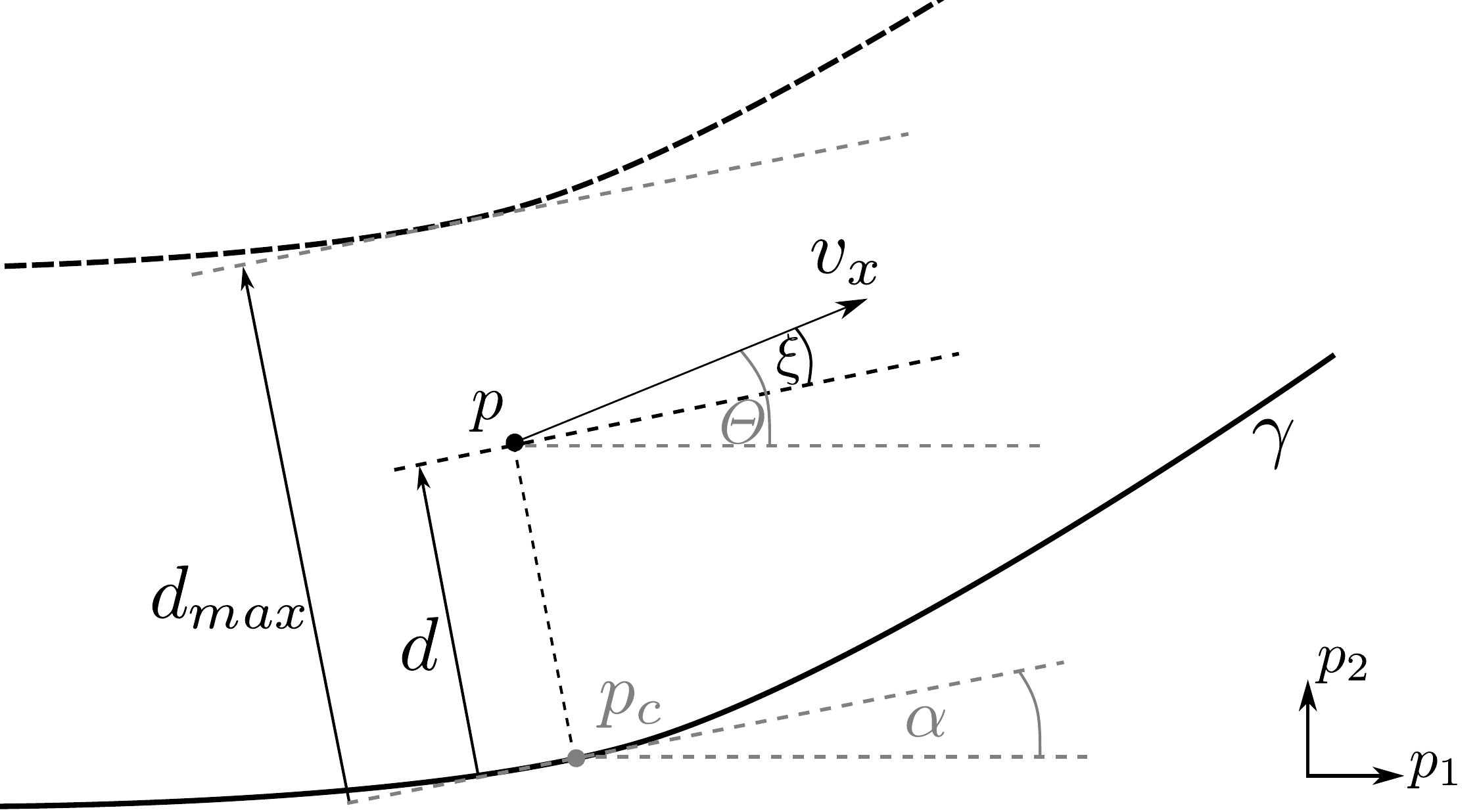} \\ (b)}
	\caption{(a) Bicycle model for the approximation of the vehicle dynamics. (b) Coordinates relative to the track center line linearized in $p_c$.}
	\label{fig:Car_Bicycle}
\end{figure}

We now want to control the vehicle such that it follows a given track $\gamma$ (e.g., a race track such as in \cite{PL14}) both securely and fast. 
In order to apply our EMOMPC algorithm, we additionally have to take the track $\gamma$ (which is now an additional parameter) into account.
Using these quantities, the first objective (i.e., security) is measured via the distance to the center line and the second objective (i.e., fast driving) is calculated via the driven distance along the track. For both objectives, we use projections of the vehicle onto the centerline:
\begin{equation}\label{eq:proj}
	\Pi_{\gamma}(p(t)) = \arg \min_{p_c \in \gamma} \|p(t) - p_c\|_2,
\end{equation}
and the corresponding distance is defined as
\begin{equation*}
	d(t) = \min_{p_c \in \gamma} \|p(t) - p_c\|_2 = \|p(t) - \Pi_{\gamma}(p(t)) \|_2.
\end{equation*}
In order to evaluate the second objective, we cannot simply use the driven distance of the vehicle due to the constant velocity $v_x$. Instead, we are interested in the driven distance with respect to the center line:
\begin{equation*}
	\int_{\Pi_{\gamma}(p(t_0))}^{\Pi_{\gamma}(p(t_e))}1~ds,
\end{equation*}
which is the standard curve integral $\int_{\gamma}f(s)~ds$ with a constant function. 
Using the above definitions, we obtain the following parameter-dependent MOCP:
\begin{equation*}
	\begin{aligned}
		\min_{x\in \mathcal{X}, u \in \SetU} J(x,u,\gamma) &= \min_{u \in \SetU} \left(\begin{array}{c}
		\int_{t_0}^{t_e} d(t)^2~dt \\
		-\int_{\Pi_{\gamma}(p(t_0))}^{\Pi_{\gamma}(p(t_e))}1~ds
		\end{array}\right) \\
		\mbox{s.t.}\qquad\qquad\qquad&\mbox{Dynamics}~\eqref{eq:ODE_Bicycle},\qquad\qquad\qquad\qquad\qquad~\\
		&\quad d \leq d_{max},
	\end{aligned}
\end{equation*}
which can be reformulated according to our framework as
\begin{equation}\label{eq:MOCP_bicycle}
	\begin{aligned}
		\min_{u \in \SetU} \hat{J}(x_0,u,\gamma) &= \min_{u \in \SetU} \left(\begin{array}{c}
		\int_{t_0}^{t_e} d(t)^2~dt \\
		-\int_{\Pi_{\gamma}(p(t_0))}^{\Pi_{\gamma}(p(t_e))}1~ds
		\end{array}\right) \\
		d &\leq d_{max}.
	\end{aligned}
\end{equation}

We now identify the symmetry action under which Problem \eqref{eq:MOCP_bicycle} is invariant, i.e., under which the objectives, the dynamics and the constraint are invariant (according to Theorem~\ref{prop:symMOCP}). The following proposition states that we can shift and rotate the ``entire problem setup''.

\begin{proposition}\label{prop:invariances_bicycle}
	Problem \eqref{eq:MOCP_bicycle} with a given reference track $\gamma$ is invariant under the group action $(\psi_g,\xi_l)$ defined by 
	\begin{equation}
	\psi_g(x) = Q\cdot x + \Delta x,\quad \xi_l(\gamma) = R_{\Delta \varTheta} \cdot \gamma + \Delta p
	\end{equation}
	with $Q = \left(\begin{matrix} R_{\Delta \varTheta} & \mathbf{0}_{2\times 3} \\  \mathbf{0}_{3\times 2} & I_{3\times 3} \end{matrix}\right)$, $R_{\Delta \varTheta} = \left(\begin{matrix} \cos{\Delta\varTheta} & -\sin{\Delta\varTheta} \\ \sin{\Delta\varTheta}&\cos{\Delta\varTheta} \end{matrix}\right)$, $\Delta x = (\Delta p_1, \Delta p_2, \Delta \varTheta, \mathbf{0}_{2,1})^{\top}$ and $\Delta p = (\Delta p_1, \Delta p_2)^{\top}$. This action represents identical translations in the vehicle's position $p$ and track $\gamma$ and simultaneous translation in orientation and rotation of the position vector for both the vehicle and the track.
	The corresponding symmetry group is thus given by $G=L=SE(2)$ with $\text{dim}(G)=\text{dim}(L)=3$.
\end{proposition}
\begin{proof}
	\underline{Equivariance of vector field:}
	With $f$ given by the right hand side of \eqref{eq:ODE_Bicycle}, we have
	\begin{align*}f(\psi_g(x),u) &= \left(\begin{array}{c}
		v_x \cos(\varTheta + \Delta \varTheta) - v_y \sin(\varTheta+ \Delta \varTheta) \\
		v_x \sin(\varTheta+ \Delta \varTheta) + v_y \cos(\varTheta+ \Delta \varTheta) \\
		r \\
		C_1 v_y + C_2 r + C_3 u, \\
		C_4 v_y + C_5 r + C_6 u
		\end{array}\right) \\ &= \left(\begin{array}{c}
		\cos{\Delta \varTheta}( v_x \cos{\varTheta} - v_y \sin{\varTheta}) - \sin {\Delta \varTheta} ( v_x \sin{\varTheta} + v_y \cos{\varTheta}) \\
		\sin{\Delta \varTheta}( v_x \cos{\varTheta} - v_y \sin{\varTheta}) + \cos {\Delta \varTheta} ( v_x \sin{\varTheta} + v_y \cos{\varTheta}) \\
		r \\
		C_1 v_y + C_2 r + C_3 u, \\
		C_4 v_y + C_5 r + C_6 u
		\end{array}\right) \\
		&= Q \cdot f(x,u) = D_x\psi_g(f(x,u)).
	\end{align*}
	\underline{Invariance of cost functions and constraints:}
	Let us assume that the curve $\gamma$ is parametrized by the paramter $\tau\in [a,b]\subset\R$, i.e., $\gamma: [a,b] \rightarrow \R^2, \tau \mapsto \gamma(\tau)$. We can then reformulate the minimal distance $d(t)$ between positon $p(t)$ and the track $\gamma$ as
	\[
		d(t) = \min_{p_c \in \gamma} \| p(t)-p_c\|_2 = \min_{\tau\in [a,b]} \| p(t)-\gamma(\tau)\|_2
	\]
	If we apply the group action $(\psi_g,\xi_l)$ to $x$ and $\gamma$ and set $\tilde{x} = \psi_g(x)$ and $\tilde{\gamma} = \xi_l(\gamma)$, we obtain
	\begin{equation}\label{eq:inv_d}
		\tilde{d}(t) := \min_{p_c \in \tilde{\gamma}} \| \tilde{p}(t) -p_c\|_2 = \min_{\tau\in [a,b]} \| R \cdot p(t) + \Delta p - (R \cdot \gamma(\tau) + \Delta p)\|_2 = \min_{\tau\in [a,b]} \| p(t)-\gamma(\tau)\|_2 = d(t)
	\end{equation}
	due to orthogonality of $R$. This means that the minimal distance between the vehicle's position and the track is invariant under the Lie group action $(\psi_g,\xi_l)$ and thus, the first objective function $\int_{t_0}^{t_e} d(t)^2\, dt$ shares the same invariance property. To show the invariance of the second objective function, let us define $\tau^* := \arg \min_{\tau \in [a,b]} \| p-\gamma(\tau)\|_2$ such that $\gamma(\tau^*) = \Pi_\gamma(p)$. For the projection in \eqref{eq:proj} applied to the $\tilde{x}$ and $\tilde{\gamma}$, it follows
	\begin{align*}
		\Pi_{\tilde{\gamma}}(\tilde{p}) &= \arg \min_{p_c\in \tilde{\gamma}} \| \tilde{p} - p_c \|_2 = \tilde{\gamma} (\arg \min_{\tau\in [a,b]} \| R p+ \Delta p - (R \gamma(\tau) + \Delta p) \|_2)\\
		& \stackrel{(*)}{=}  \tilde{\gamma} (\arg \min_{\tau\in [a,b]} \| p - \gamma(\tau)  \|_2) = \tilde{\gamma}(\tau^*) = R \gamma(\tau^*) + \Delta p = R \Pi_\gamma(p) + \Delta p,
	\end{align*}
	i.e., the projection of the symmetry transformed problem is the symmety transformed projection of the original problem. 
	Let us define $\tau^*_0 = \arg \min_{\tau \in [a,b]} \| \tilde{p}(t_0)-\tilde{\gamma}(\tau) \|_2$ and $\tau^*_e = \arg \min_{\tau \in [a,b]} \| \tilde{p}(t_e)-\tilde{\gamma}(\tau) \|_2$. From the equality $(*)$, it can easily be seen that $\tau_0^*$ and $\tau_e ^*$ are also the minimizers of the original projection problem, i.e.~$\tau^*_0 = \arg \min_{\tau \in [a,b]} \| {p}(t_0)-{\gamma}(\tau) \|_2$ and $\tau^*_e = \arg \min_{\tau \in [a,b]} \| {p}(t_e)-{\gamma}(\tau) \|_2$. Furthermore, we observe that $\tilde{\gamma}'(\tau)  = \frac{d}{d\tau} \tilde{\gamma}(\tau) = \frac{d}{d\tau} R {\gamma}(\tau) + \Delta p = R \gamma'(\tau)$.
	Application of the Lie group action $(\psi_g,\xi_l)$ to the second objective function then yields
	\begin{align*}
		\int_{\Pi_{\tilde{\gamma}}(\tilde{p}(t_0))}^{\Pi_{\tilde{\gamma}}(\tilde{p}(t_e))} 1 \ ds &= \int_{\tilde{\gamma}(\tau_0^*)}^{\tilde{\gamma}(\tau_e^*)} 1 \ ds = \int_{\tau^*_0}^{\tau^*_e} \| \tilde{\gamma}'(\tau)\|_2 \ d\tau =  \int_{\tau^*_0}^{\tau^*_e} \| R {\gamma}'(\tau)\|_2 \ d\tau = \int_{\tau^*_0}^{\tau^*_e} \|  {\gamma}'(\tau)\|_2 \ d\tau \\ &= \int_{\Pi_{{\gamma}}({p}(t_0))}^{\Pi_{{\gamma}}({p}(t_e))} 1 \ ds,
	\end{align*}
	which shows the invariance of the second objective function. Finally, due to the invaraince of $d(t)$ (Equation \eqref{eq:inv_d}), the constraint $d \leq d_{max}$ is invariant under the Lie group action as well. The statement follows with Theorem~\ref{prop:symMOCP} and Remark~\ref{remark:equif}.
\end{proof}

\begin{remark}[Numerical approximation of $\gamma$]\label{rem:parametrization_track}
	Proposition~\ref{prop:invariances_bicycle} allows us to reduce both $n_x$ and $n_\gamma$ by three since $\mbox{dim}(G) = \mbox{dim}(L) = 3$. However, $\gamma$ is infinite-dimensional for arbitrary tracks. Therefore, we will use a local approximation for $\gamma$ in the numerical realization. To this end, we approximate the track on the prediction horizon by fixing the curvature $\kappa = \frac{d \alpha}{ds} \big|_{p_c}$ with $\alpha(s=0)$ being the angle between the track and the horizontal axis in $p_c$, cf.~Figure~\ref{fig:Car_Bicycle}~(b):
	\[
		\gamma(s) = \left(\begin{array}{c}
		c_1 \\ c_2
		\end{array}\right) + \frac{1}{\kappa} \left(\begin{array}{c}
		\cos(s - c_3) \\ \sin(s - c_3)
		\end{array}\right).
	\]
	Here, $c \in \R^3$ is determined in such a way that $\gamma(0) = p_c$ and $\frac{d \gamma}{ds} \big|_{s=0} = \alpha(0)$. 
	The symmetry group $L$ allows us to reduce the dimension of the parametrization by three, i.e., we can shift and translate $\gamma$ in such a way that $\gamma(0) = (0,0)$ and $\alpha(0) = 0$, which results in the following explicit formulation:
	\[
		\gamma(s) = \frac{1}{\kappa} \left(\begin{array}{c}
		\cos(s - \pi/2) \\ 1+ \sin(s - \pi/2)
		\end{array}\right).
	\]
\end{remark}

The consequence of Proposition~\ref{prop:invariances_bicycle} and Remark~\ref{rem:parametrization_track} is a significant reduction in the dimension of the parametrization of the offline phase. From the dynamics, there only remain $v_y$ and $r$. The invariances with respect to the track result in $\kappa$ as the only parameter. However, as we have only invariance with respect to identical translation and rotation of track and vehicle, we still have to take the relative position and orientation into account, i.e., the distance $d$ and the angle $\xi$ between track and vehicle direction (cf.~Figure~\ref{fig:Car_Bicycle}~(b)). In total, this results in a five-dimensional parametrization:
\[
	(\widetilde{x}_0, \widetilde{\gamma}) = (v_y, r, \xi, d, \kappa)^\top,
\]
and the initial orientation and position for the vehicle dynamics become:
\[
	\varTheta(t_0) = \xi, \quad p(t_0) = (0, d)~\mbox{with}~ \Pi_\gamma(p(t_0)) = (0,0).
\]
Regarding the parametrized track (i.e., $p_c$, $\alpha$, $\kappa$), this is a reduction from nine to five parameters.

\begin{remark}
	Although it does not fit into the Lie group setting, we can further reduce the computational effort by a factor of two using the observation that
	\[
		\arg \min_{u} \hat{J}(x_0,u,\gamma) = -\arg \min_{u} \hat{J}(-x_0,u,-\gamma),
	\]
	which corresponds to a reflection at the horizontal axis. Consequently, we only need to consider deviations between vehicle and center line to one side (e.g., to the left side).
\end{remark}

\subsubsection{Offline phase}

In the offline phase, we first construct the library $\mathcal{L}$ and then solve Problem~\eqref{eq:MOCP_bicycle} for each entry. To this end, we implement a direct approach and discretize the problem such that we obtain a nonlinear MOP. We choose $v_x = 30$, $t_0 = 0~\mathsf{sec}$, $t_e = 0.5~\mathsf{sec}$, and a time step of $h = 0.05~\mathsf{sec}$. Consequently, we have $u\in[u_{min},u_{max}]^{10}$, where $u_{min} = -0.5$ and $u_{max} = 0.5$.
The library $\mathcal{L}$ is built according to the bounds and step sizes in Table~\ref{tab:Library}, which leads to a total number of $223,587$ MOCPs that we have to solve according to Algorithm~\ref{alg:MOCP_Auto}. The number of targets is set to $n_T = 18$ such that each Pareto front is approximated by $20$ points. Note that without exploiting the symmetry, the number of MOCPs would be four to five orders larger, which would lead to a prohibitively large CPU cost. The present number is still very high, but the solution is obtained in approximately two days when running the computation in parallel on $72$ cores, which is acceptable since the offline phase has to be performed only once. 

\begin{table}[t]
\caption{Parameters for the library $\mathcal{L}$.}
\label{tab:Library}
\centering
\begin{tabular}{ccccc}
	\hline
	Variable & Minimal value & Maximal value & Step size & Number of grid points \\ 
	\hline
	$v_y$ & $-3$ & $3$ & $0.5$ & $13$ \\ 
	$r$ & $-6$ & $6$ & $1$ & $13$ \\ 
	$\xi$ & $0$ & $10$ & $0.5$ & $21$ \\ 
	$d$ & $-\pi/4$ & $\pi/4$ & $\pi/12$ & $7$ \\ 
	$\kappa$ & $-0.1$ & $0.1$ & $0.025$ & $9$ \\ 
	\hline
\end{tabular}
\end{table}

The solution of one such MOCP is shown in Figure~\ref{fig:Bicycle_Offline}, where the resulting vehicle trajectories are depicted in (a) and the corresponding Pareto front in (b). We see that the front is non-convex such that a simple solution method such as weighted sum would not be capable of computing all Pareto optimal solutions.

\begin{figure}[h!]
	\centering
	\parbox[b]{0.45\textwidth}{\centering \includegraphics[width=0.37\textwidth]{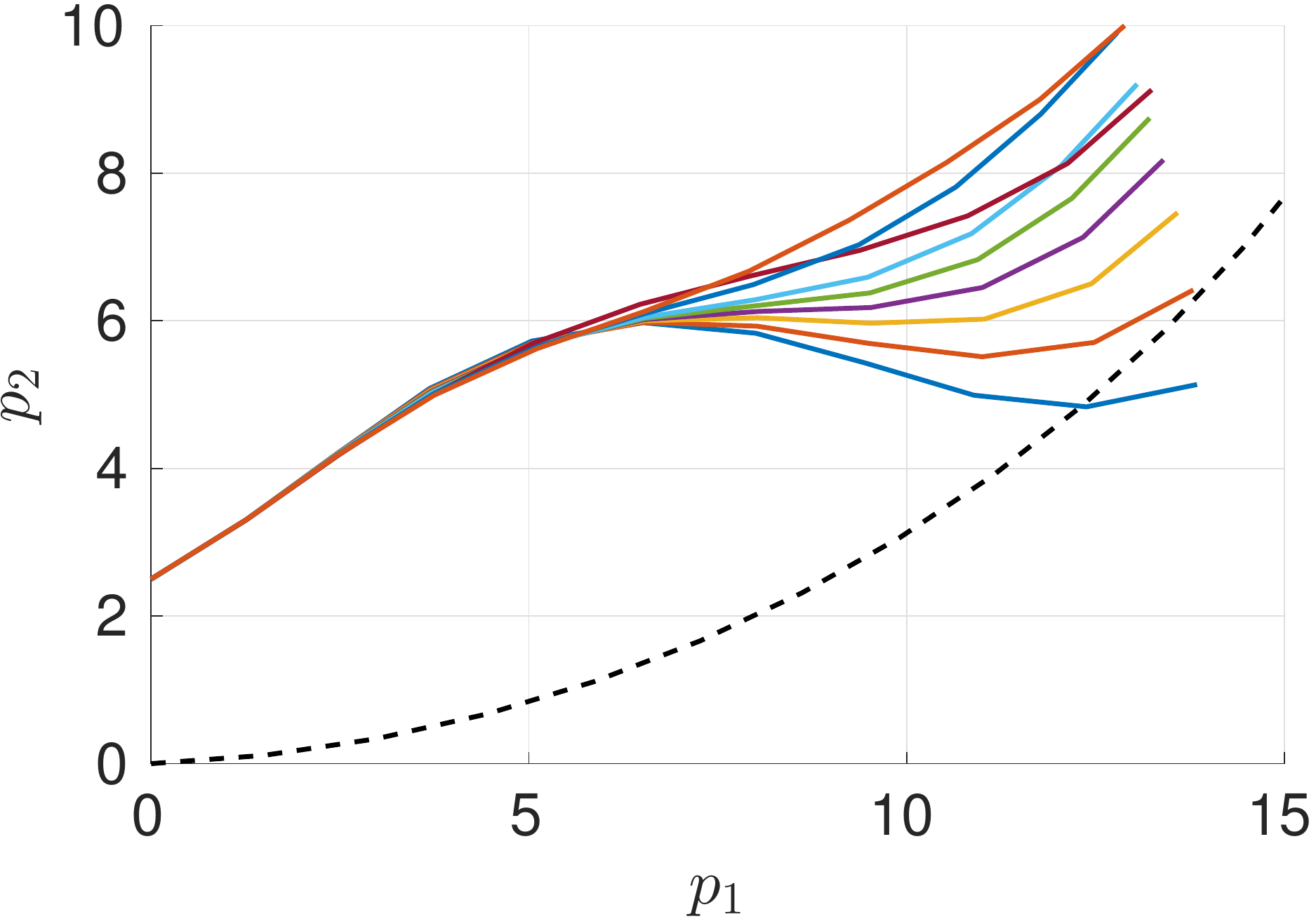} \\ (a)}\hfil
	\parbox[b]{0.45\textwidth}{\centering \includegraphics[width=0.37\textwidth]{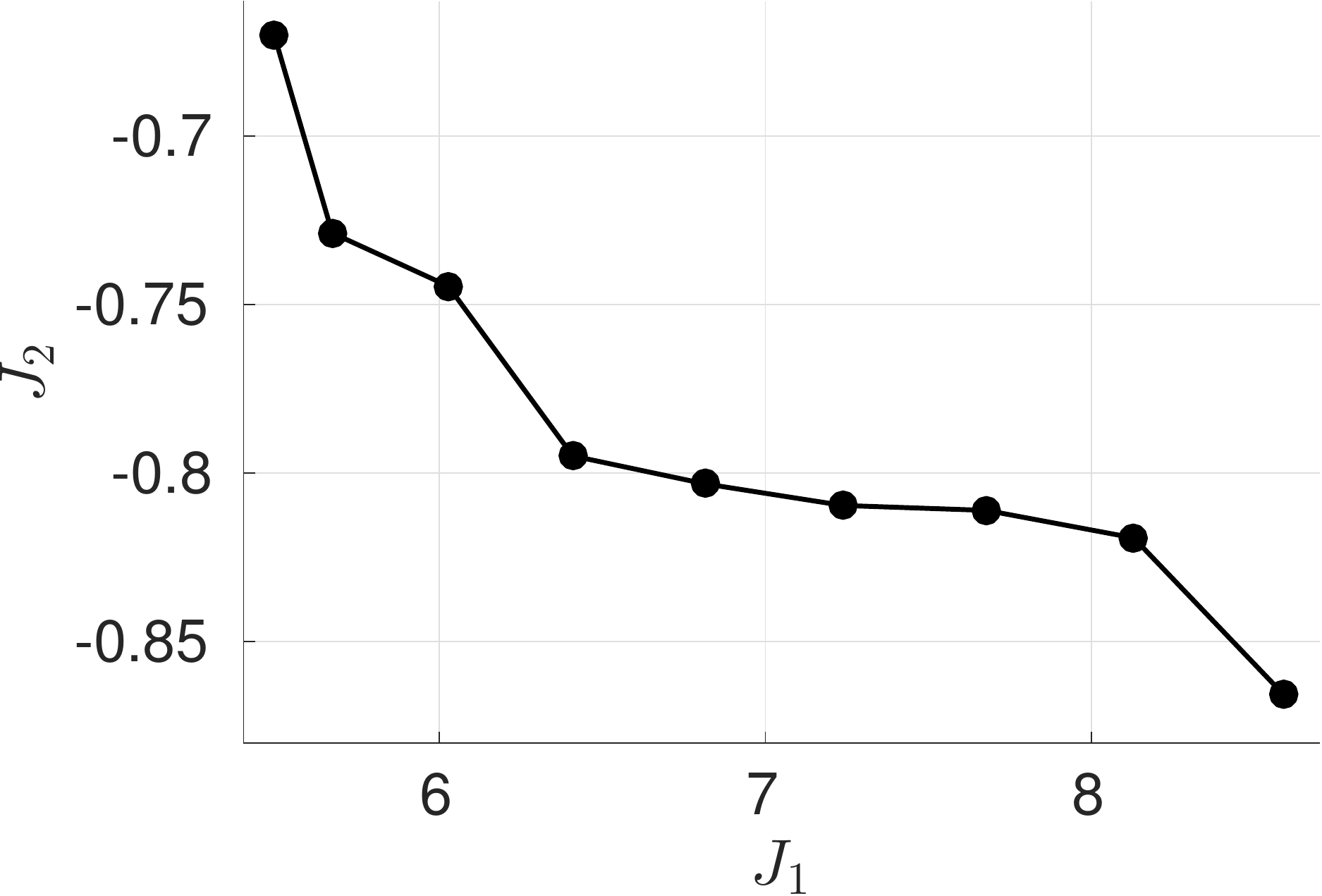} \\ (b)}
	\caption{MOCP for the parameter value $(\widetilde{x}_0, \widetilde{\gamma}) = (0, 1, \pi/6, 2.5, 0.05)^\top$. (a) The resulting Pareto optimal trajectories, where $\rho$ is increasing from blue to orange. The reference track is shown in black. (b) The corresponding Pareto front.}
	\label{fig:Bicycle_Offline}
\end{figure}

\subsubsection{Online phase}
In the online phase, the MPC loop is realized according to Algorithm~\ref{alg:online}. We test our EMOMPC framework on two test tracks with different curvature values. They can be seen in Figure~\ref{fig:Bicycle_Tracks}, where the second track is a scaled version of the first one shown in (a). We select $t_c = h$ as the control horizon and formulate global objectives that are meaningful with respect to driving one lap. Consequently, the objective of driving as far as possible is transformed to driving one lap as fast as possible. The second objective remains unchanged, i.e., we want to drive as close to the centerline as possible.

We begin by fixing the weight $\rho$ for the entire track, where the solutions for $\rho = 1$ (fast) and $\rho = 0.25$ (close to the center line) are shown in Figure~\ref{fig:Bicycle_Tracks}~(a) and (b). In Figure~\ref{fig:Bicycle_ParetoFronts}~(a), the resulting lap time and integrated distance to the center line are shown (i.e., the two objectives fast versus safe driving but for the entire track). As can be seen, the typical trade-off behavior between the objectives is carried over from the offline phase to the entire track, which yields precisely the desired additional control freedom for which the multiobjective setup is introduced in the first place. Nevertheless, we also see that very low weights (i.e., close to the center line) even lead to an increase in the driven distance. The reason is very likely that there is no regularization term in the first objective which penalizes the control cost. In combination with the error that is introduced by the discretized library and the resulting interpolation, we observe a zig zag behavior around the center line which results in increased distances and lap times. The effect can also be seen in Figure~\ref{fig:Bicycle_Offline}~(a), where we have a crossing of the center line for $\rho = 0$. For the track with higher curvature, this effect is even more apparent. Consequently, it is advisable to restrict the choice of $\rho$ to a subset of $[0,1]$. Moreover, the implementation of an adaptive library construction similar to \cite{Joh02,BF06} is a promising direction for future work.

\begin{figure}[t]
\centering
\parbox[b]{0.45\textwidth}{\centering \includegraphics[width=0.4\textwidth]{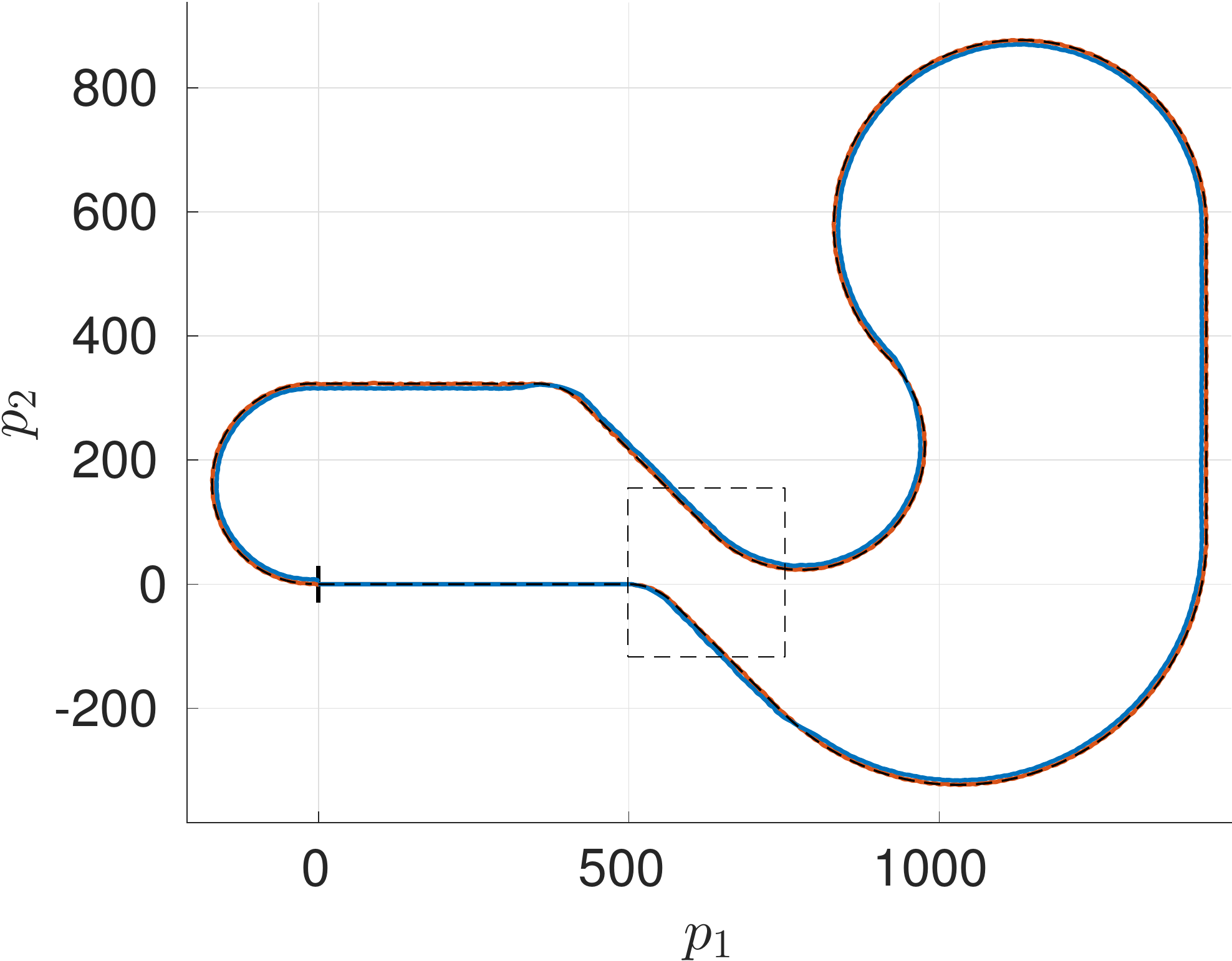} \\ (a)}\hfil
\parbox[b]{0.45\textwidth}{\centering \includegraphics[width=0.4\textwidth]{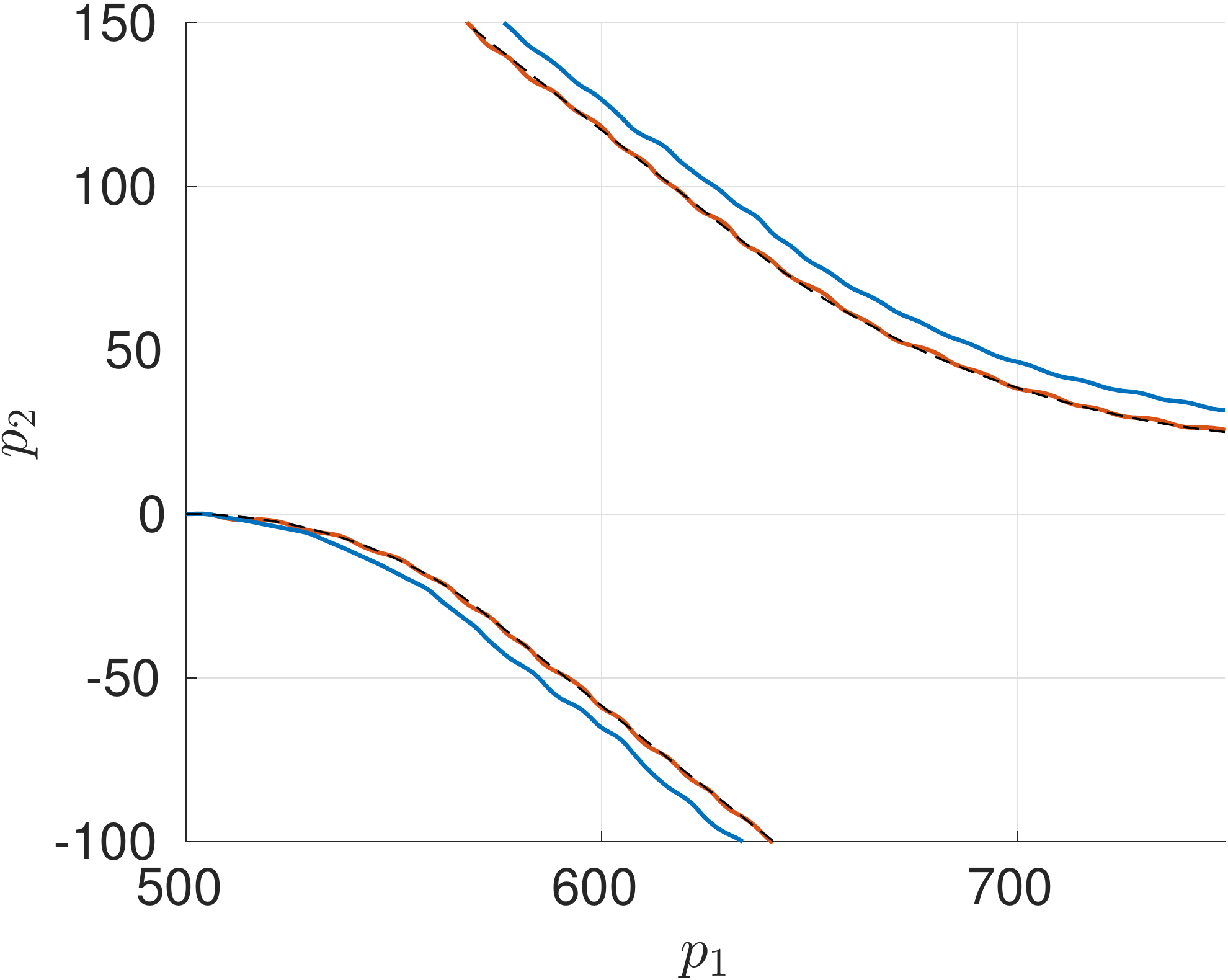} \\ (b)} \\
\parbox[b]{0.45\textwidth}{\centering \includegraphics[width=0.4\textwidth]{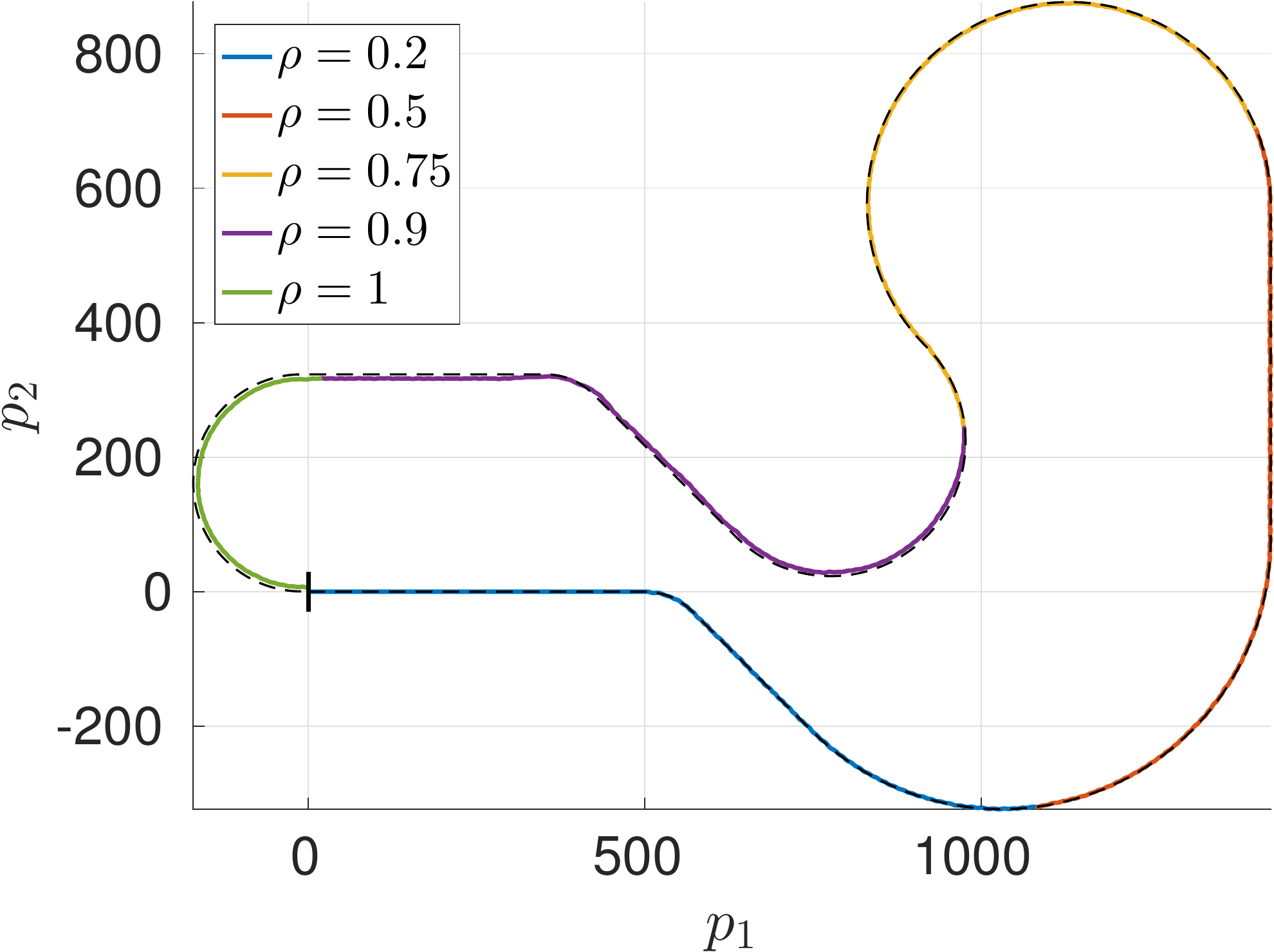} \\ (c)}\hfil
\parbox[b]{0.45\textwidth}{\centering \includegraphics[width=0.4\textwidth]{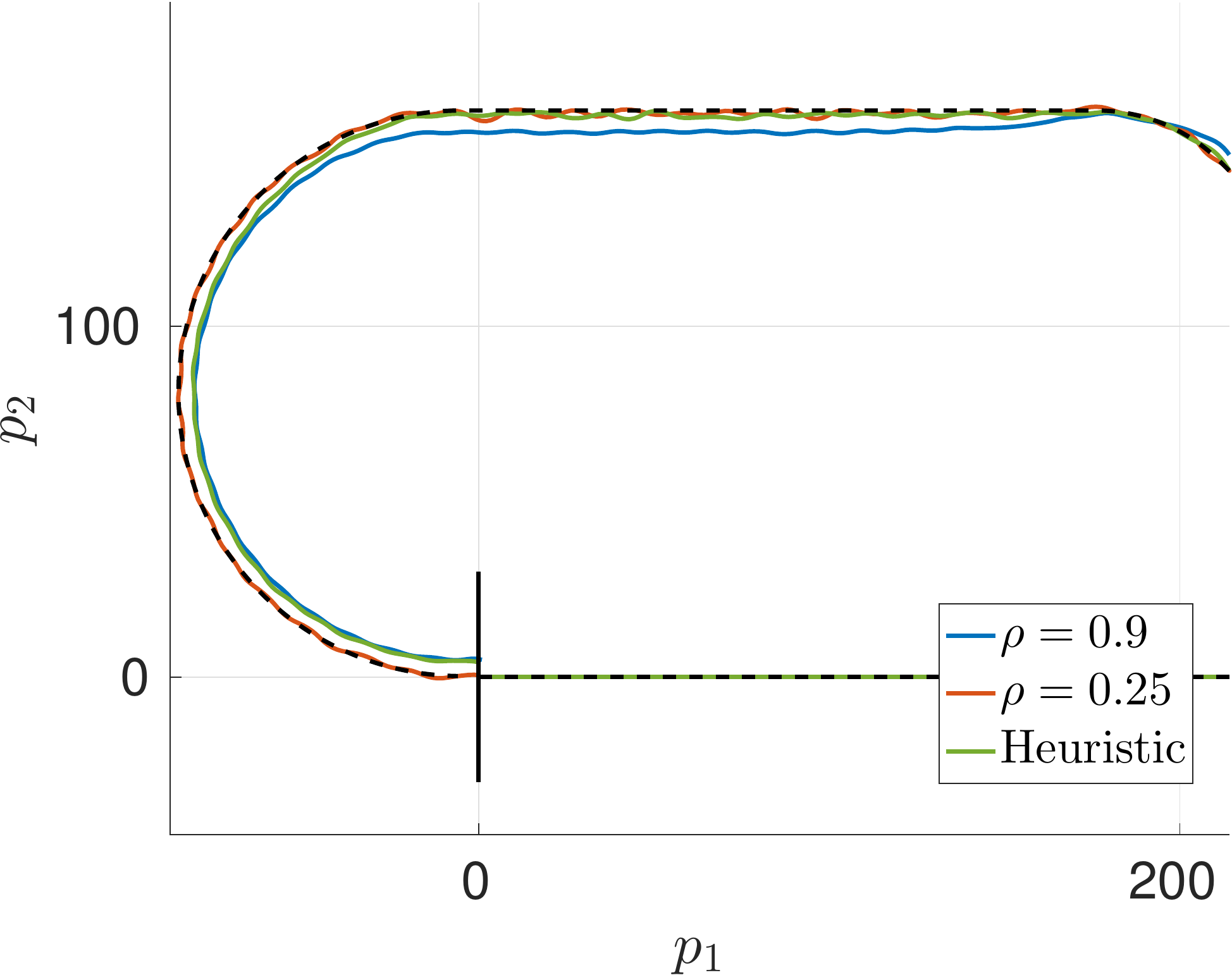} \\ (d)} 
\caption{Test track. (a) Solution of the EMOMPC algorithm for constant weights (blue: $\rho= 1$, orange: $\rho = 0.25$). (b) Zoom of the dashed box in (a). We see that large values of $\rho$ lead to following the inside of a curve. (c) Variation of $\rho$ at constant time instances (from safe to fast driving) indicated by different colors. (d) A scaled track with higher curvature for comparison. We see that low values of $\rho$ lead to a zigzag behavior around the center line. The green curve is obtained by the heuristic \eqref{eq:EMOMPC_Heuristic}.}
\label{fig:Bicycle_Tracks}
\end{figure}

A key advantage of the EMOMPC algorithm is that we are capable of adjusting the weight online in order to react to a changing priorization of the objectives in an ad-hoc manner. This is visualized in Figure~\ref{fig:Bicycle_Tracks}~(c), where $\rho$ is changed at constant time instances in order to change the priorization from safe to fast driving. An alternative to allowing the decision maker to adjust $\rho$ manually is to implement a situation-based heuristic in order to realize the desired behavior. This is shown in Figure~\ref{fig:Bicycle_Tracks}~(d) for the smaller track. We here use a very simple approach where $\rho_i$ is increased or decreased (within prescribed bounds) in the $i^{\mathsf{th}}$ MPC loop according the the current curvature:
\begin{equation} \label{eq:EMOMPC_Heuristic}
	\rho_i = \begin{cases}
	\max \{0.25, \rho_{i-1} - 0.05\} & \mbox{for}~\kappa < \epsilon \\
	\min \{0.90, \rho_{i-1} + 0.05\} & \mbox{for}~\kappa \geq \epsilon
	\end{cases}.
\end{equation}
This means that we want to drive close to the center line on straight parts of the track and drive fast through curves. The performance of both the manually as well as heuristically varied weight is shown in Figure~\ref{fig:Bicycle_ParetoFronts} as well. We cannot guarantee optimality for the entire track in the MPC framework, and we see that both approaches outperform a constant weight solution. This gives a strong motivation for further investigating the choice of $\rho$.

\begin{figure}[t]
	\centering
	\parbox[b]{0.45\textwidth}{\centering \includegraphics[width=0.35\textwidth]{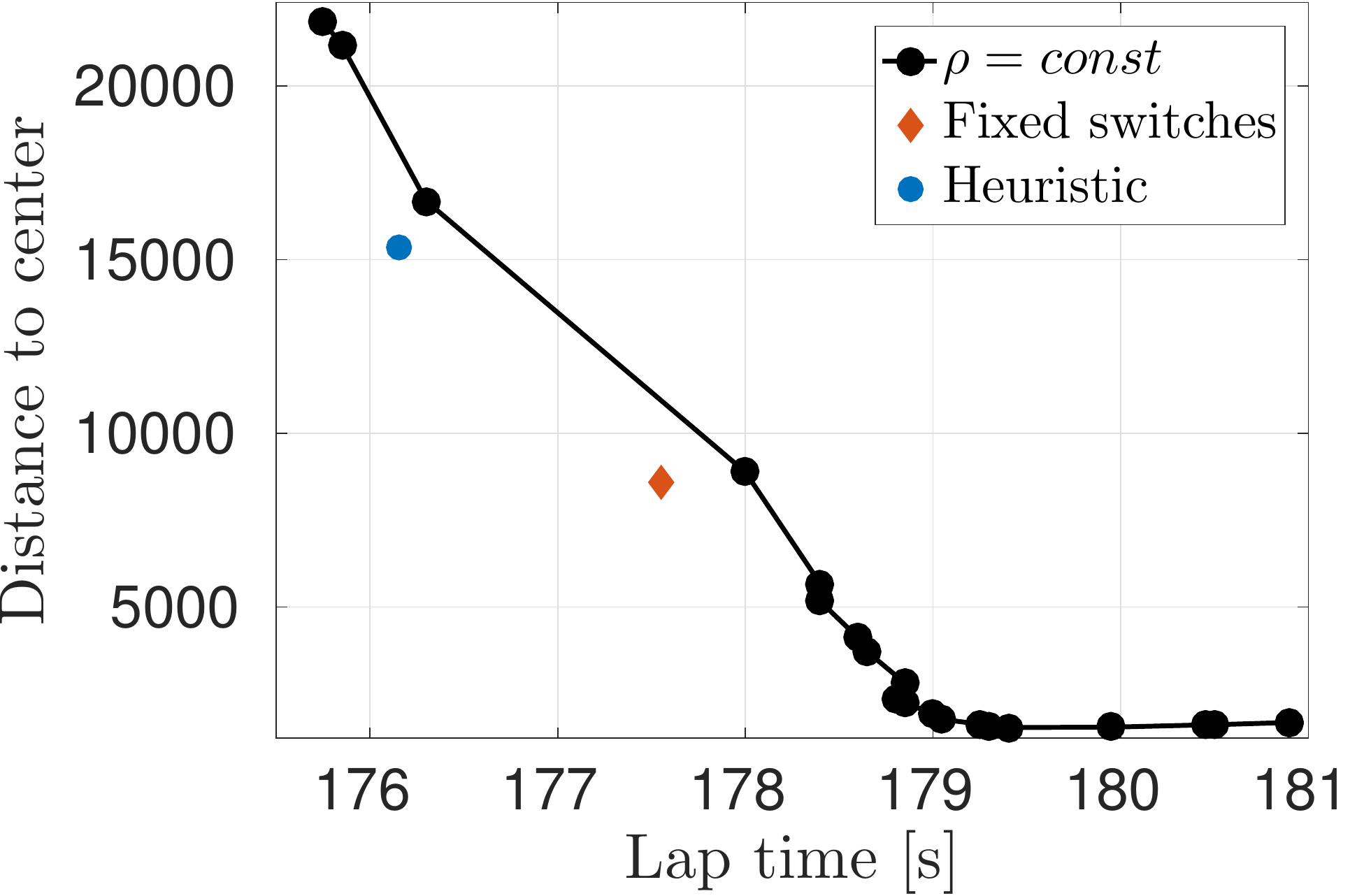} \\ (a)}\hfil
	\parbox[b]{0.45\textwidth}{\centering \includegraphics[width=0.35\textwidth]{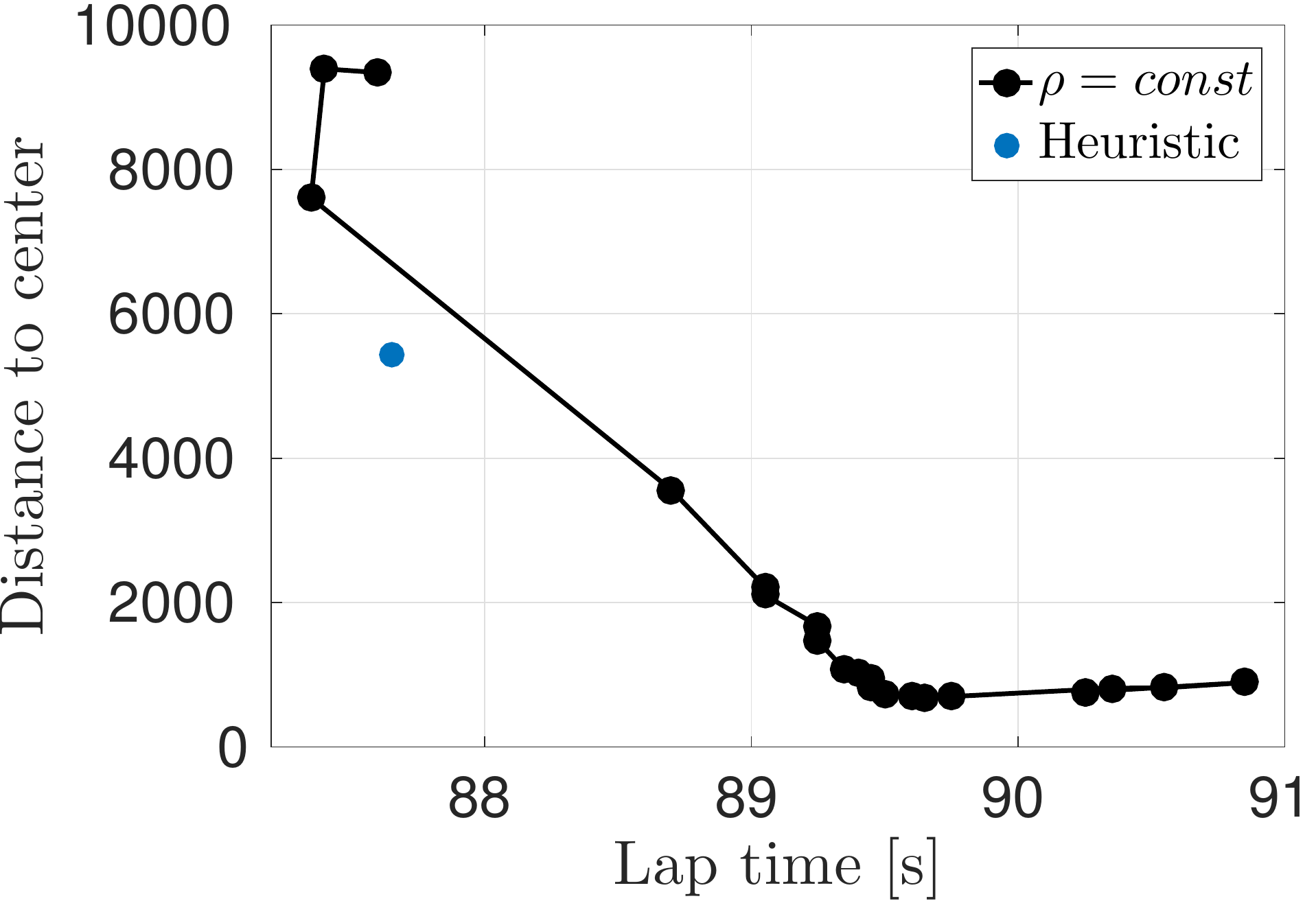} \\ (b)}
	\caption{``Pareto fronts'' (Lap time vs.~integrated distance to center line) for two different track sizes. The black lines are solutions with constant weights $\rho = (0, 0.05, 0.1, \ldots, 1)$. (a) Large track, relatively small curvature values. (b) Smaller track, i.e., higher curvature values. The blue points are obtained by applying the heuristic \eqref{eq:EMOMPC_Heuristic} and the orange point results from fixed switching points, cf.~Figure~\ref{fig:Bicycle_Tracks}~(c).}
	\label{fig:Bicycle_ParetoFronts}
\end{figure}

\subsection{Intelligent cruise control for electric vehicles}
\label{subsec:Examples_EV}
The second example is also related to autonomous driving. Here, we want to control the longitudinal dynamics of an electric vehicle for varying speed limits with respect to the objectives energy efficiency and fast driving. The electric vehicle under consideration has already been studied in various scenarios, see \cite{DEF+14} for single-objective open loop control with height profiles (obtained from GPS data) and \cite{DEF+17} for the extension to two objectives. In \cite{EPS+16}, nonlinear and linear MPC approaches have been compared and in \cite{PSOB+17}, the concept presented here was first applied in a very practical manner. Here, we pick it up once again mainly in order to illustrate the procedure of numerically identifying invariances in situations where it is difficult to show these analytically.

The system dynamics are described by a four-dimensional, nonlinear ODE for the state $x = (v,S,U_{d,L}$, $U_{d,S})$. Here, $v$ is the vehicle velocity, $S$ is the battery state of charge, and $U_{d,L}$ and $U_{d,S}$ are the long term and short term voltage drops, respectively. The system is controlled by the wheel torque $u$, and the battery current $I$ is determined via an algebraic equation. 
The model is described in detail in \cite{EPS+16}, the right-hand side contains both highly nonlinear terms and lookup-tables such that symmetries are very difficult to identify analytically. For the MPC, we have to solve the following MOCP:
\begin{equation}\label{eq:MOCP_EV}
	\begin{aligned}
		\min_{u \in \SetU} J(x_0,u,\gamma)  &= \min_{u \in \SetU} \left(\begin{array}{c}
			S(t_0) - S(t_e) \\ t_e - t_0
		\end{array}\right)  \\
		\dot{x}(t) &= f(x(t), u(t)), \\
		v_{min}(t) &\leq v(t) \leq v_{max}(t), \\
		I_{min} &\leq I(t) \leq I_{max}, \\
		x(0) &= x_0, p(t_e) = p_e.
	\end{aligned}
\end{equation}
For this problem, the parameter $\gamma$ describes the current velocity constraints depending on the part of the track. This is discussed in more detail in the following section.

\subsubsection{Offline phase}
In the offline phase, we again have to construct a library for $x_0$ and $\gamma(p) = (v_{min}(p), v_{max}(p))$, where $p = \int_{t_0}^{t_e} v(t) \ dt$ is the current position of the vehicle. Similar to the first example, $\gamma$ is a function and thus infinite-dimensional. We therefore use a linear approximation for the velocity bounds. Moreover, we distinguish between four different scenarios: (i) constant velocity, (ii) acceleration, (iii) deceleration and (iv) stopping, see Figure~\ref{fig:EV_OfflineScenarios}. In the constant velocity scenario, we simply have $\gamma = (0.8 \overline{v}, \overline{v})$, where $\overline{v}$ is the maximal allowed velocity on this part of the track. In the stopping scenario, we have a terminal constraint at position $p_f$ which is $v(p_f) = 0$. This yields $\gamma(p_f) = (0, 0)$. Finally, for the acceleration and deceleration, we introduce a lower or upper bound for the velocity derivative $dv/dp$, respectively, which is determined in each time step according to Figure~\ref{fig:EV_OfflineScenarios}~(b):
\[
	\frac{dv}{dp} = a(p) \lessgtr \frac{v(p_f) - v(p)}{p_f - p}.
\]
This leads to $\gamma(p) = (v(p(t_0)) + a p, \infty)$ for the acceleration and $\gamma = (-\infty, v(p(t_0)) + a p)$ for the deceleration case.

\begin{figure}[h!]
	\centering
	\parbox[b]{0.3\textwidth}{\centering \includegraphics[width=0.28\textwidth]{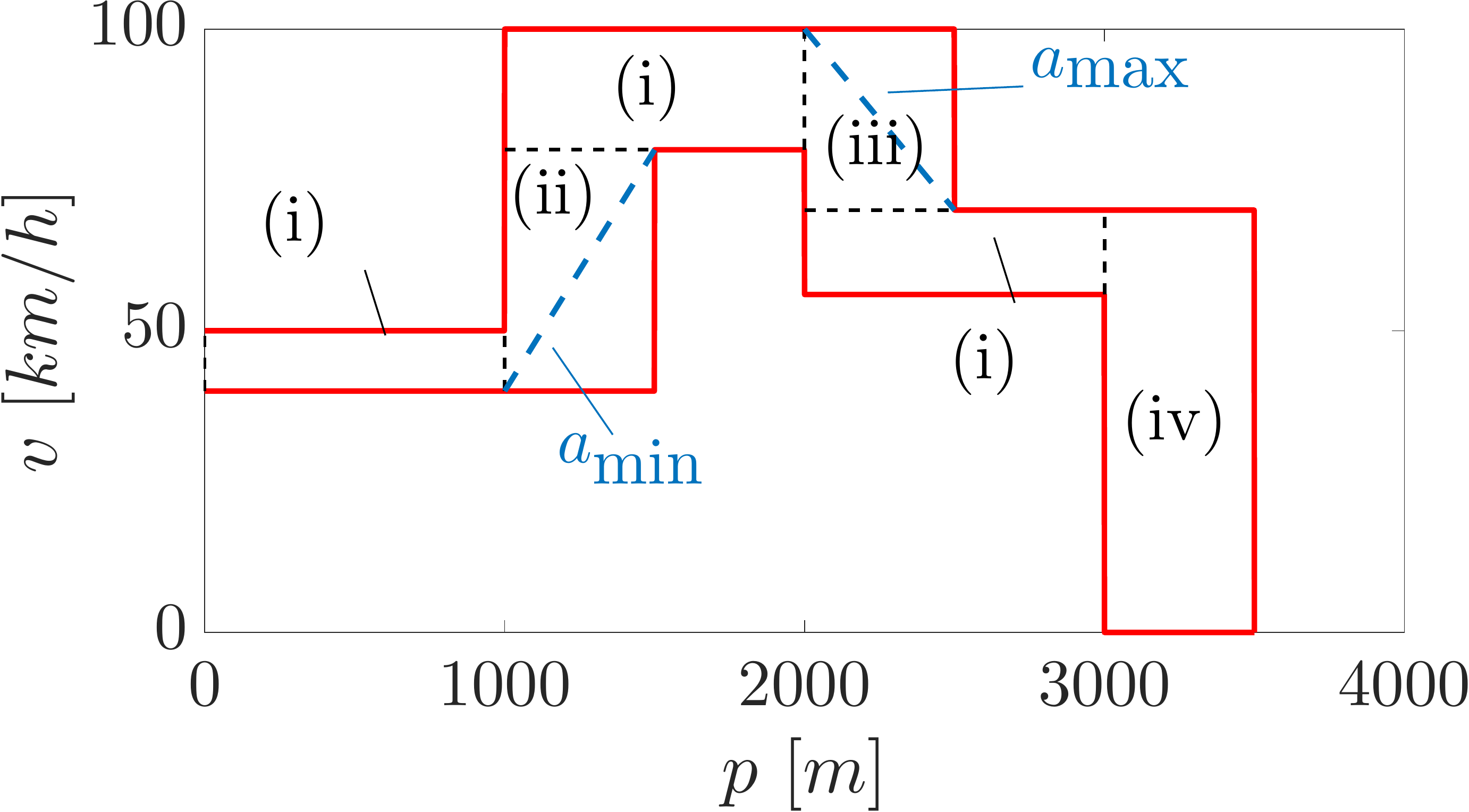} \\ (a)}\hfil
	\parbox[b]{0.3\textwidth}{\centering \includegraphics[width=0.28\textwidth]{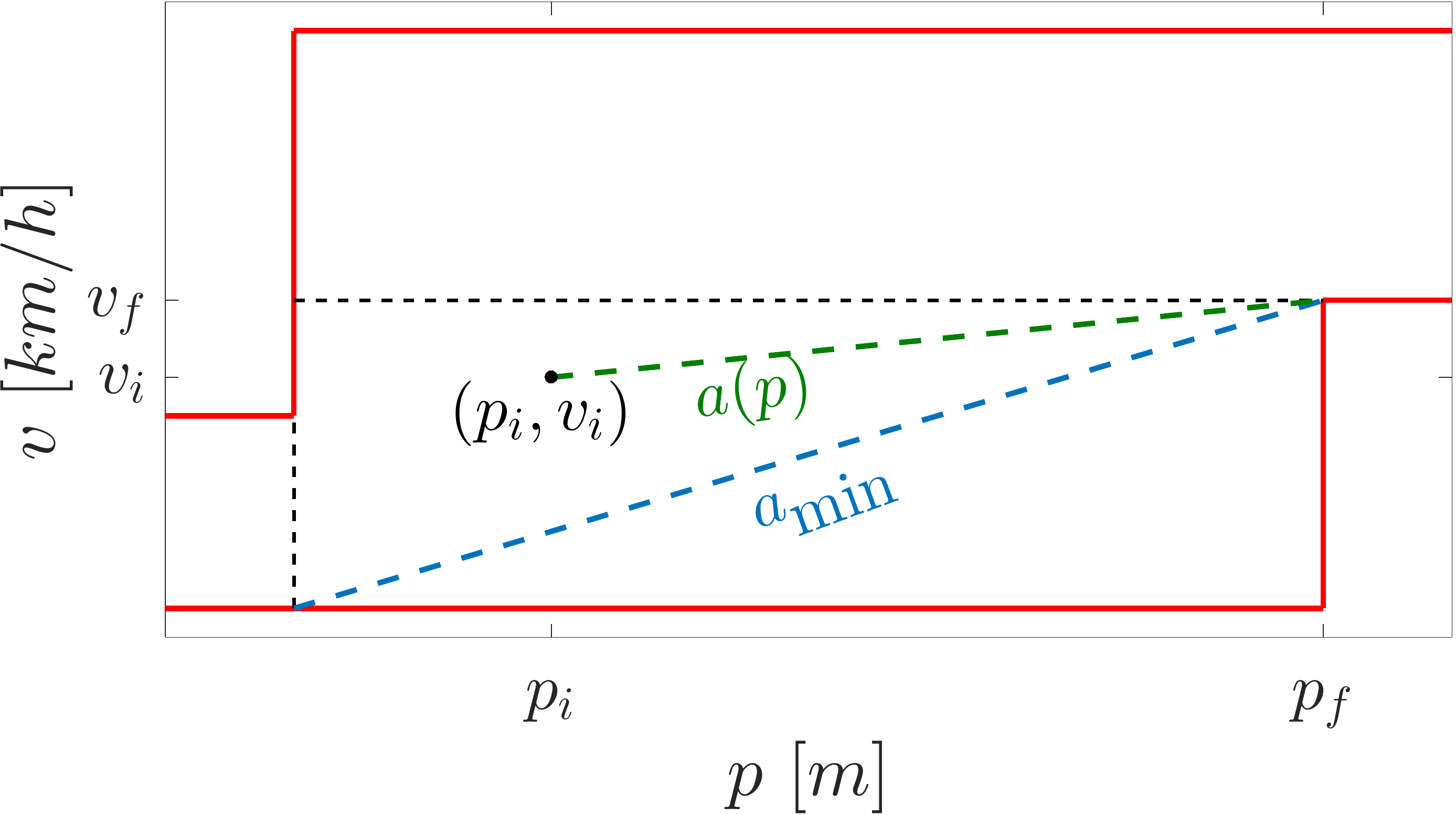} \\ (b)}
	\caption{\cite{PSOB+17} Constraints which are encoded in $\widetilde{\gamma}$ according to the respective scenario. (a) Four scenarios: (i) constant velocity ($\widetilde{\gamma} = \overline{v}$), (ii) acceleration ($\widetilde{\gamma} = a$), (iii) deceleration ($\widetilde{\gamma} = a$), (iv) stopping ($\widetilde{\gamma} = p_f$). (b) Bound $a(p)$ for the acceleration scenario. }
	\label{fig:EV_OfflineScenarios}
\end{figure}

In order to reduce the dimension of $x_0$, we numerically investigate the dependence of the solution on the different states. This is shown in Figure~\ref{fig:EV_OfflineInvariances}, where we see in (a) and (b) that the state of charge and the vehicle velocity are almost invariant with respect to the initial value of $S$ (for $S(t_0) > 5 \%$). On the other hand, we see in (c) that the system is not invariant with respect to $v(t_0)$. Proceeding this way with all variables, we see that problem \eqref{eq:MOCP_EV} is (almost) invariant with respect to translations in $S(t_0)$ as well as $U_{d,L}(t_0)$ and $U_{d,S}(t_0)$. The objective function is invariant as a consequence of the invariance of $S$ and $v$ and numerical tests show that $I$ is also almost invariant with respect to translations in these quantities, i.e., the constraints are also invariant. This results in $\widetilde{x}_0 = v(t_0)$. Consequently, we only have to construct a library for $(\widetilde{x}_0, \widetilde{\gamma}) = (v(t_0), \widetilde{\gamma})$, where $\widetilde{\gamma} \in \{\overline{v}, a,p_f\}$ depending on the current scenario. After the appropriate discretization, this results in 1727 MOCPs for the offline phase.
\begin{figure}[h!]
	\centering
	\parbox[b]{0.3\textwidth}{\centering \includegraphics[width=0.3\textwidth]{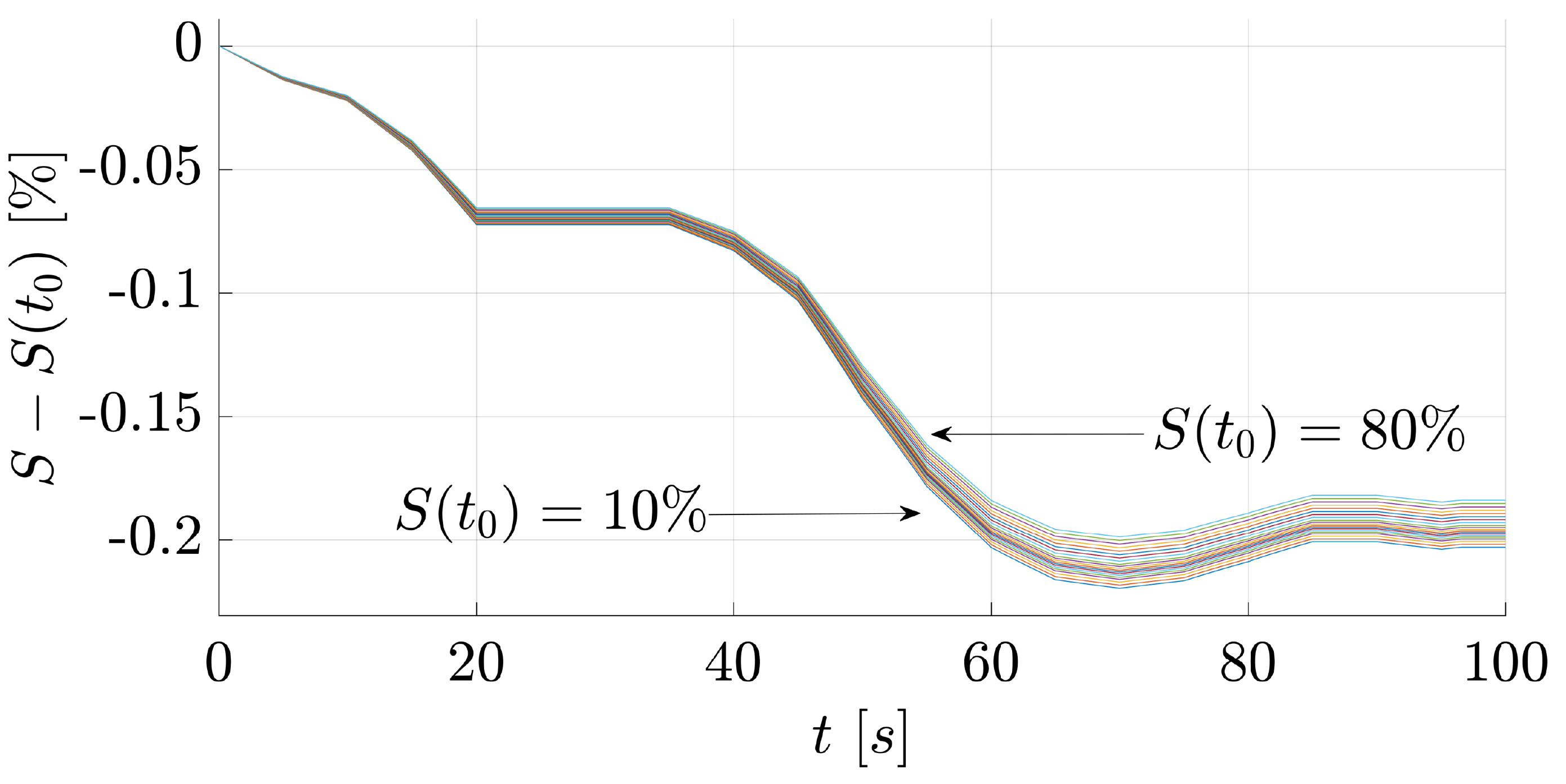} \\ (a)}\hfil
	\parbox[b]{0.3\textwidth}{\centering \includegraphics[width=0.3\textwidth]{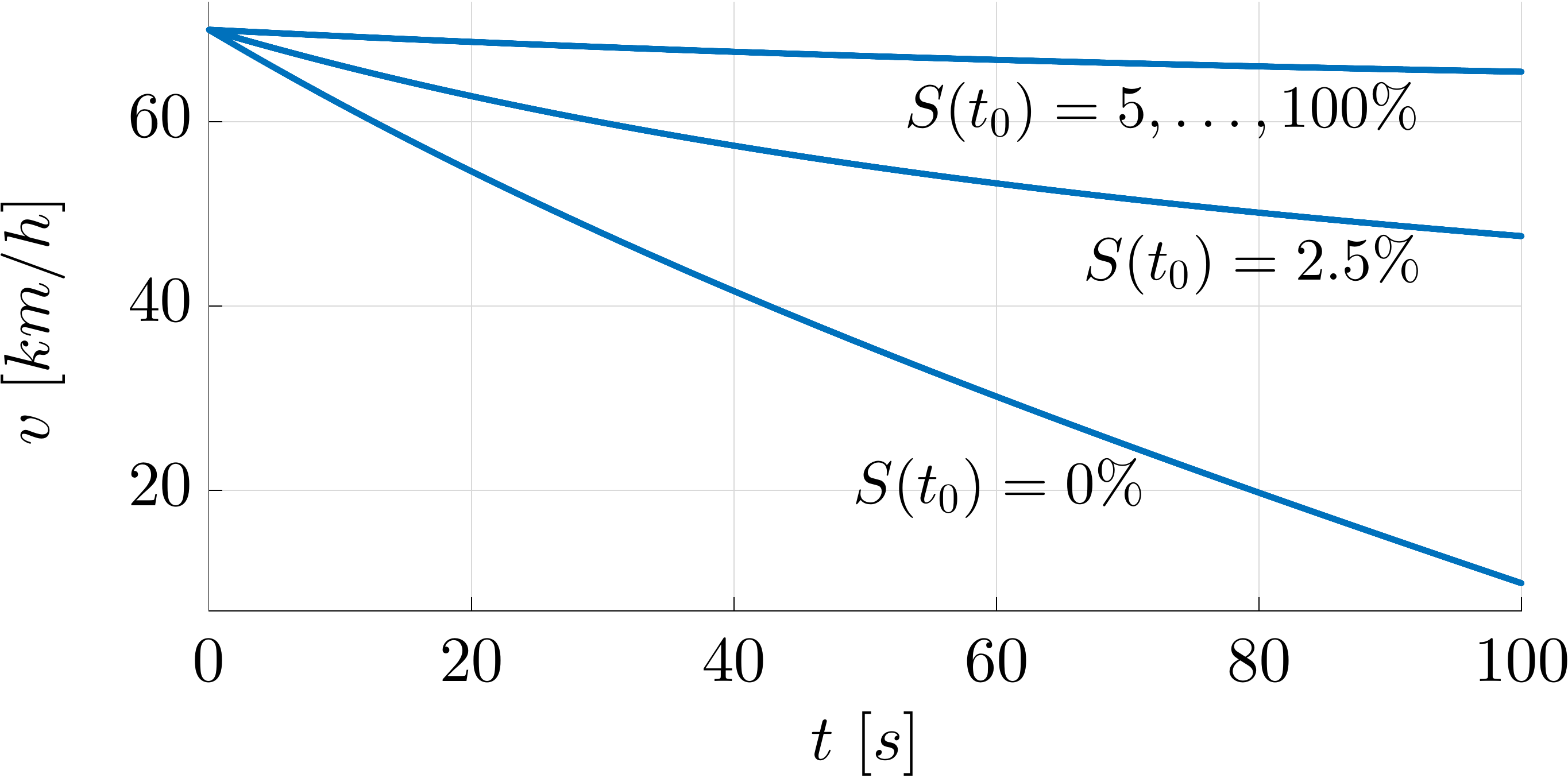} \\ (b)}\hfil
	\parbox[b]{0.3\textwidth}{\centering \includegraphics[width=0.3\textwidth]{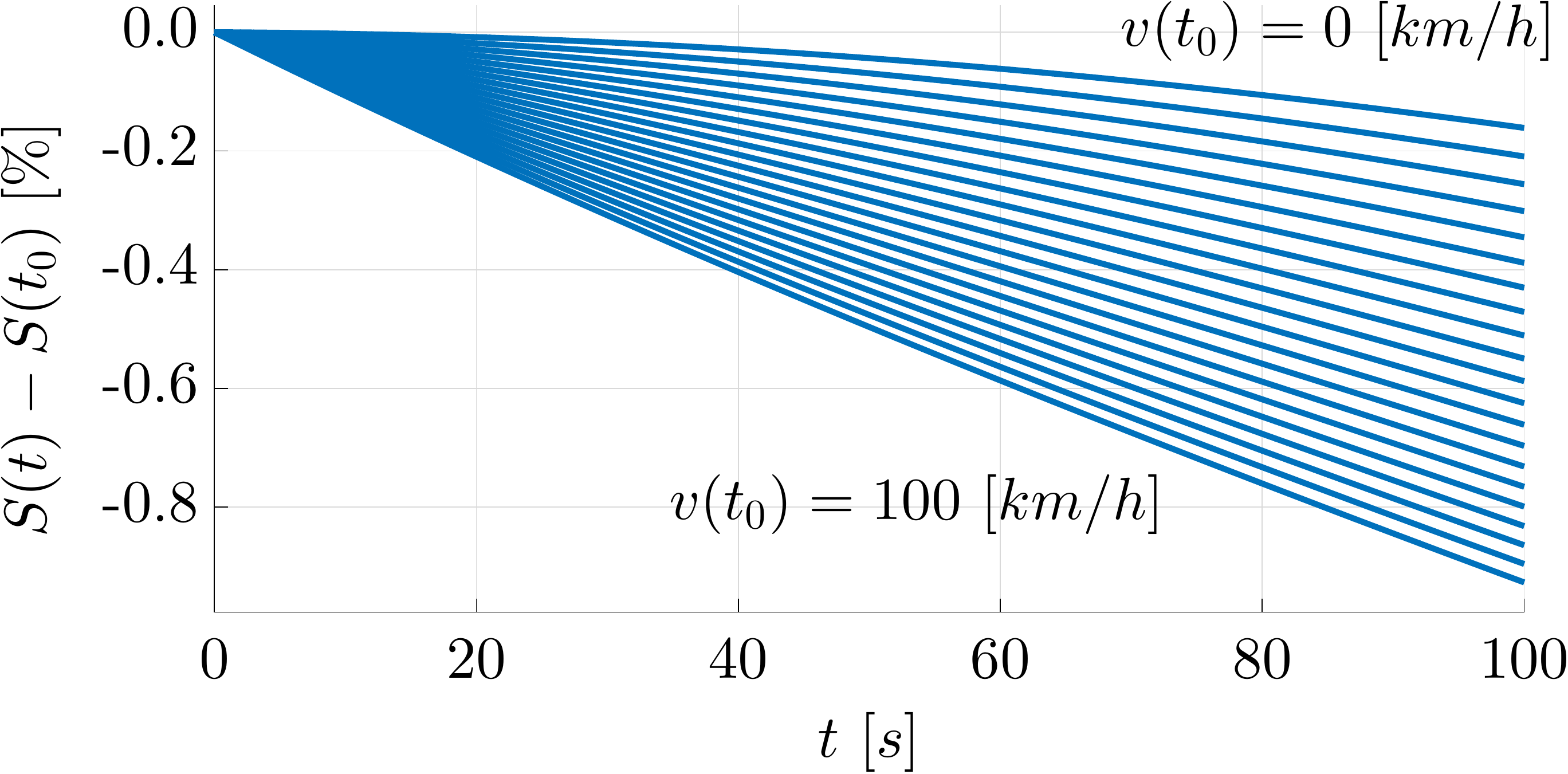} \\ (c)}
	\caption{\cite{PSOB+17} (a) Almost invariance of the state of charge with respect to the initial state of charge $S(t_0)$. (b) Almost invariance of the velocity with respect to the initial state of charge $S(t_0)$. (c) No invariance with respect to the initial velocity.}
	\label{fig:EV_OfflineInvariances}
\end{figure}

\subsubsection{Online phase}
The online phase follows precisely Algorithm~\ref{alg:online}, i.e., we identify $(\widetilde{x}_0, \widetilde{\gamma})$ and select (or interpolate) the corresponding Pareto set from the library. According to the passenger's preference $\rho$, we then apply the control to the plant and repeat the process at the next time instance. As mentioned in Section~\ref{sec:Symmetries}, it is not necessary to store the state trajectories but only the Pareto optimal controls. The results are shown in Figure~\ref{fig:EV_Online}. In (a), several trajectories for constant weights (dashed black lines) and one for a varying weight $\rho$ (green, changes at $p=2000$ and $p=4000$) are shown. We observe that the constraints are not violated (by construction of the library) and that moreover, the algorithm enables us to vary the behavior of the nonlinear vehicle dynamics in a very flexible manner by varying $\rho$. In Figure~\ref{fig:EV_Online}~(b), we see the corresponding objective function values for the entire track, where the classical trade-off behavior can be observed. In \cite{PSOB+17}, there is furthermore a discussion on automatically choosing $\rho$ in a similar fashion to the heuristic \eqref{eq:EMOMPC_Heuristic} implemented for the first example. This way, the additional flexibility of changing $\rho$ online is lost but on the other hand, a near-optimal performance can be achieved which is verified by comparing the results to a very expensive dynamic programming (DP) solution for a simplified track. This is visualized in Figure~\ref{fig:EV_Online}~(c) and (d), where the heuristic is -- simply speaking -- to select large values for $\rho$ at low velocities, low values at higher velocities, and additionally perform slight adjustments when approaching switches between different scenarios. We see that when appropriately choosing $\rho$, we can achieve almost globally optimal behavior in real-time.

\begin{figure}[h!]
	\centering
	\parbox[b]{0.45\textwidth}{\centering \includegraphics[width=0.38\textwidth]{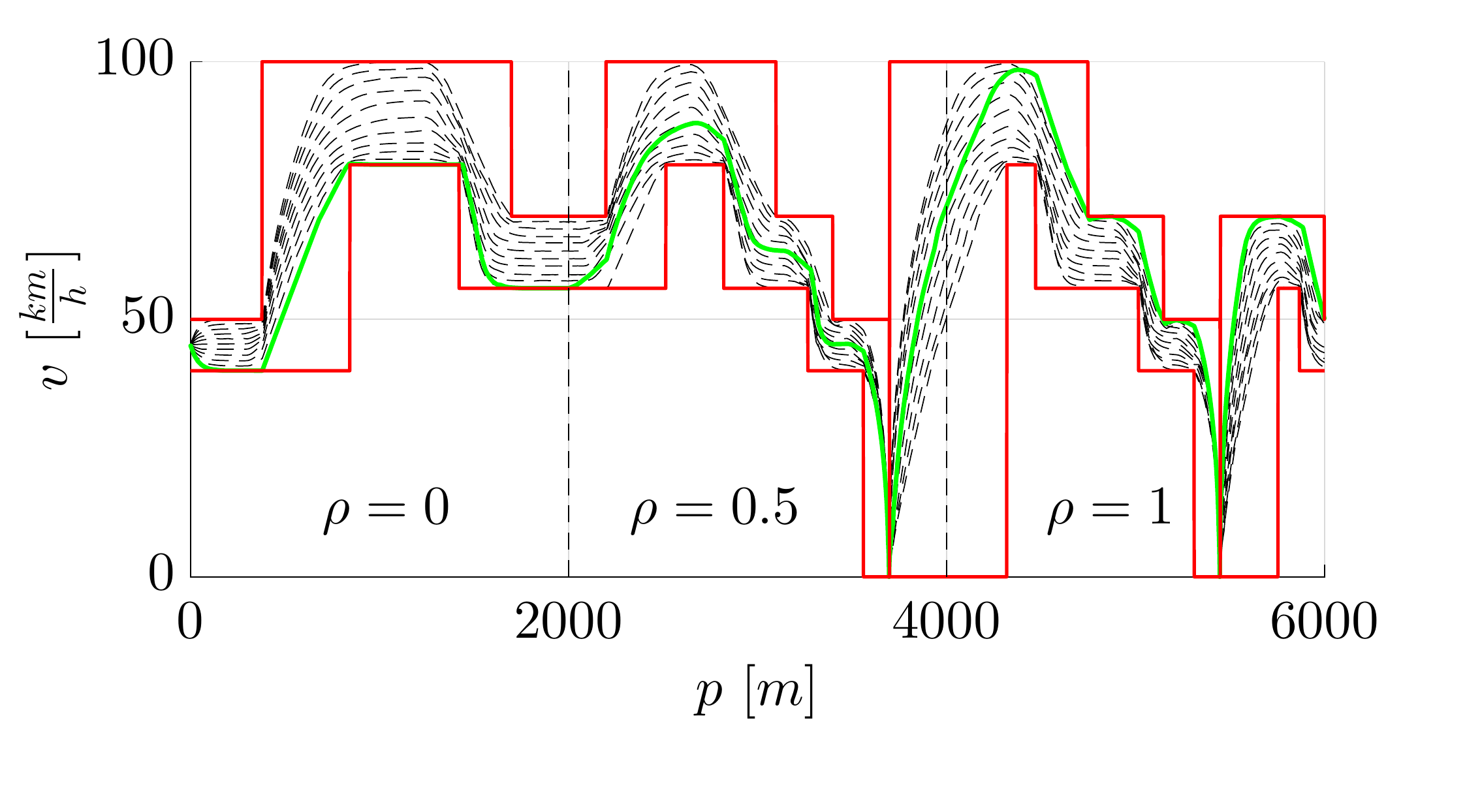} \\ (a)}\hfil
	\parbox[b]{0.4\textwidth}{\centering \includegraphics[width=0.32\textwidth]{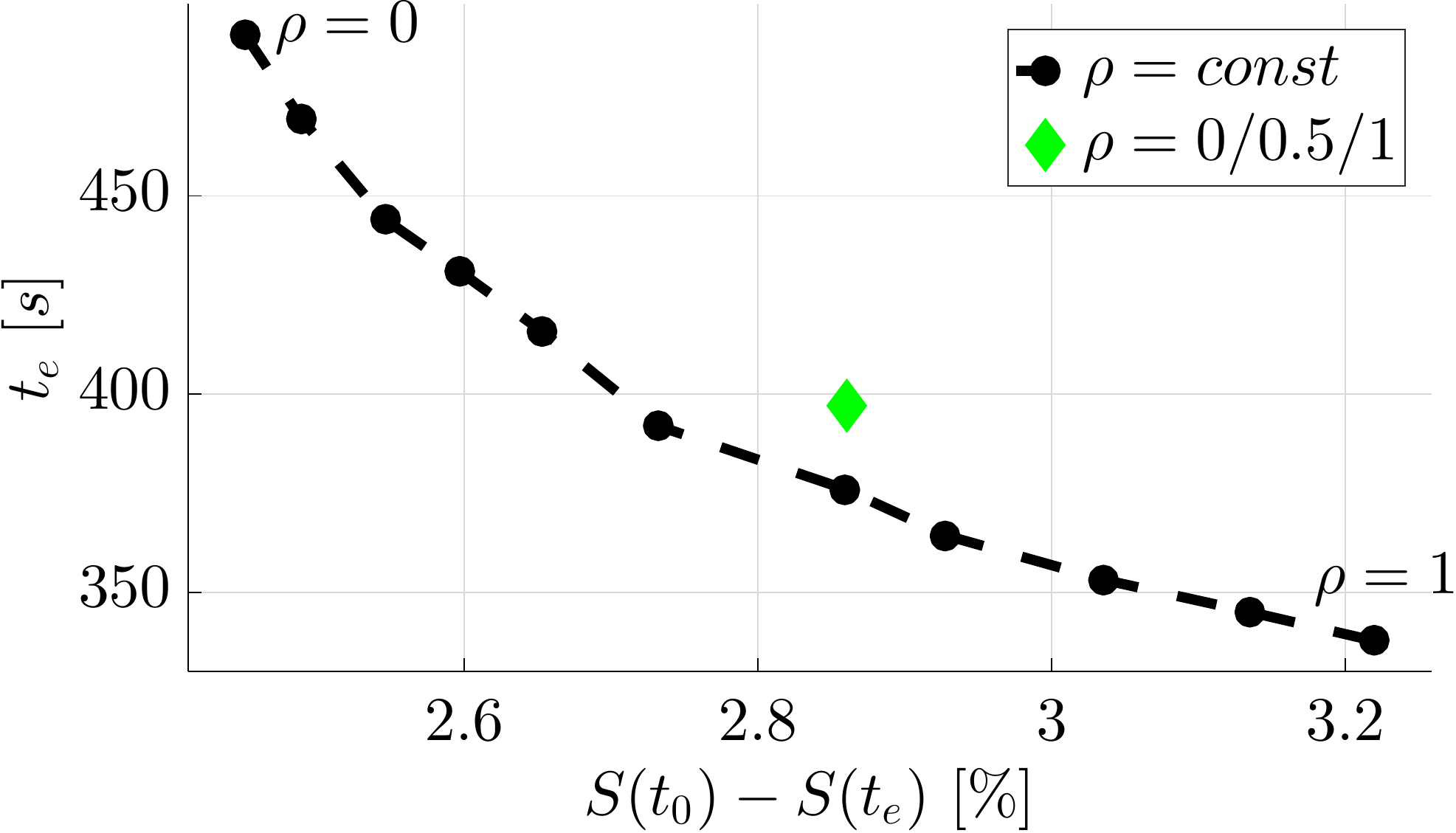} \\ (b)} \\
	\parbox[b]{0.45\textwidth}{\centering \includegraphics[width=0.38\textwidth]{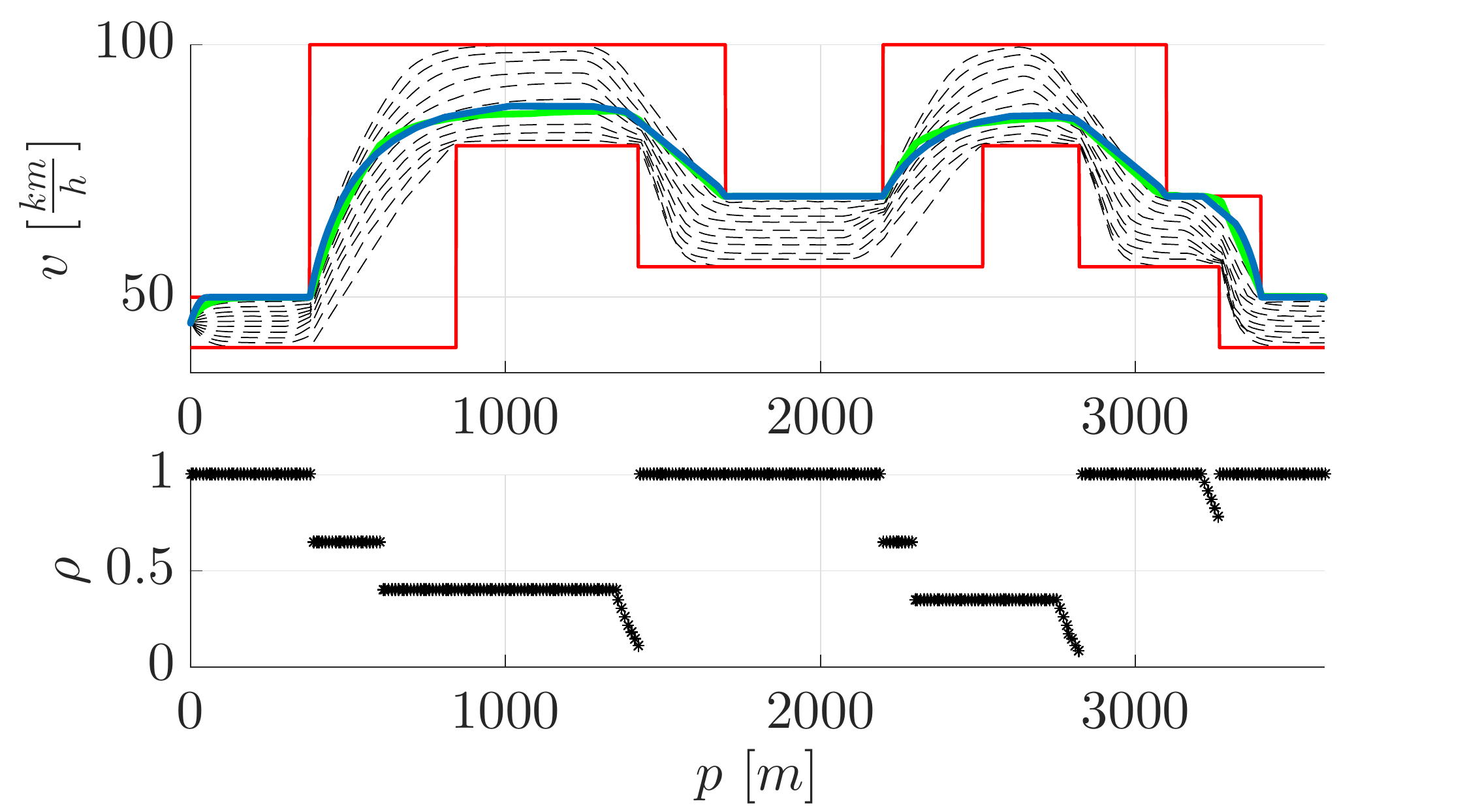} \\ (c)}\hfil
	\parbox[b]{0.4\textwidth}{\centering \includegraphics[width=0.32\textwidth]{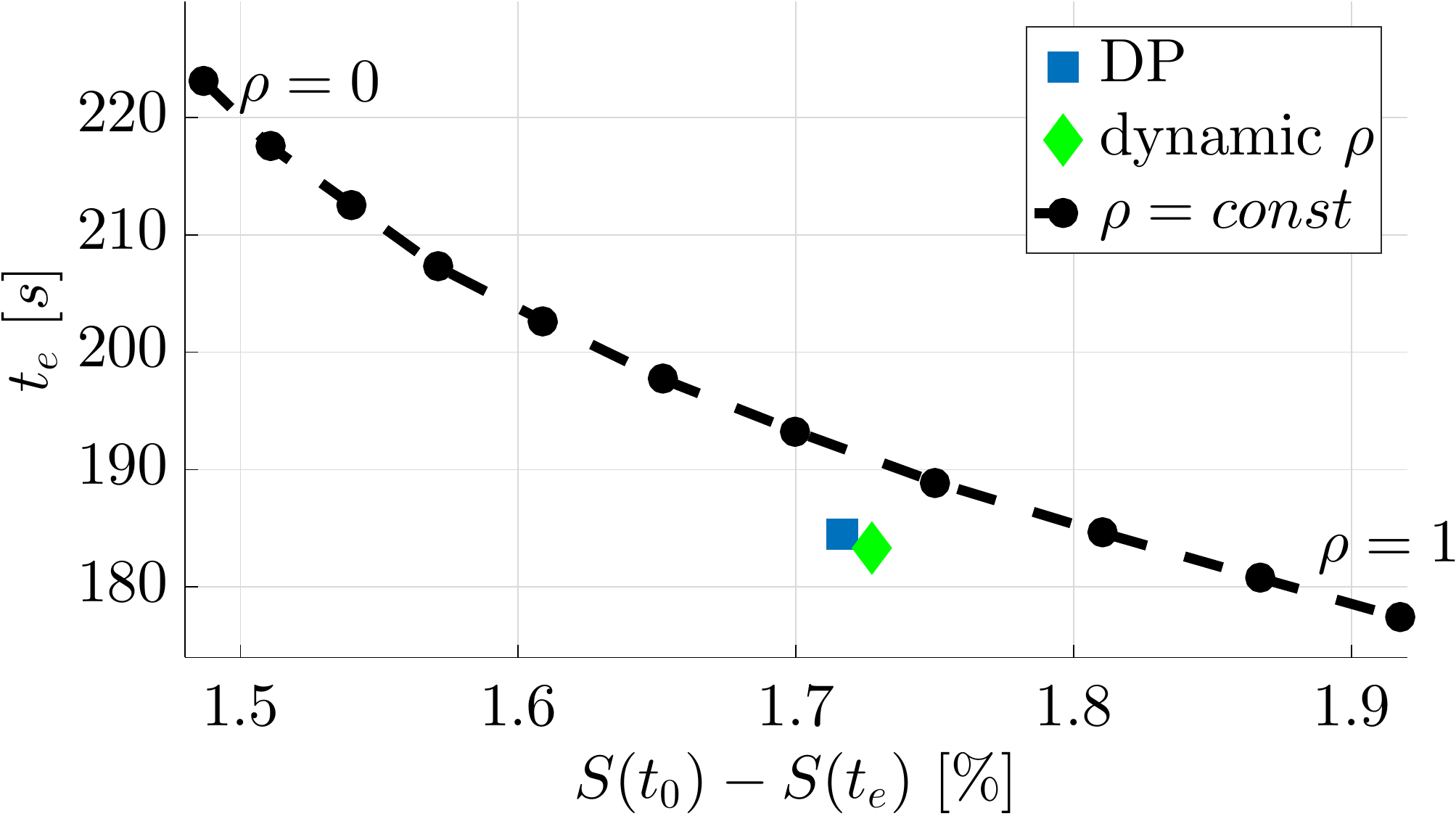} \\ (d)}
	\caption{\cite{PSOB+17} (a) Application of the EMOMPC algorithm to an example track. Dashed lines: constant weights $\rho$. Green line: varying weight $\rho$. (b) The Pareto front corresponding to (a). (c) Simpler track and heuristically chosen $\rho$ (see bottom plot) and comparison to global optimum obtained via dynamic programming. (d) The ``Pareto front'' corresponding to (c).}
	\label{fig:EV_Online}
\end{figure}


\section{Conclusion}
\label{sec:Conclusion}

We have presented an explicit MPC algorithm for nonlinear dynamical systems with multiple objectives which extends the results from \cite{BMDP02} in several regards. In order to reduce the computational effort, we make extensive use of symmetries in the dynamical control system and the multiobjective optimal control problem. In contrast to the classical approach from motion planning with motion primitives, we only require the $\arg \min$ of two problems to be identical, i.e.~Pareto sets are valid in multiple situations and we do not have to store Pareto optimal trajectories.

Two applications from autonomous driving demonstrate the efficiency and the additional control freedom this approach yields. In contrast to arbitrarily weighting different objectives, knowledge of the entire Pareto set gives the decision maker a much deeper insight and allows for a better selection of the compromise solution. Moreover, convexity cannot be guaranteed for nonlinear problems such that one is not necessarily able to compute all optimal compromises by weighting, as we have seen in the first example. By allowing variations of the objective priorization online, adaptivity to changing situations is increased in contrast to scalar-valued problem formulations. Alternatively, heuristics for choosing the weight depending on the current situation allow for the computation of nearly globally optimal solutions despite the real-time situation.

For future work, it will be interesting to develop error estimates for the interpolation procedure similar to what was presented in \cite{Joh02,BF06} for scalar-valued problems, and to interactively add elements to the library such that the maximum error is reduced. This way, the computational effort may be reduced while increasing the quality of the resulting trajectory. Furthermore, the results show that efficient heuristics for automatically choosing the weight parameter $\rho$ are worth investigating.

\paragraph{Acknowledgement}
This work has been partially funded by the EPSRC project: ``Fractional Variational Integration and Optimal Control''; ref: EP/P020402/1. The calculations were performed on resources provided by the Paderborn Center for Parallel Computing (PC$^2$). 
We thank Kathrin Fla{\ss}kamp for fruitful discussions on symmetries in control systems.

\bibliographystyle{unsrt}
{\footnotesize
\bibliography{EMOMPC}

\begin{thebibliography}{10}

\bibitem{Ehr05}
M.~Ehrgott.
\newblock {\em {Multicriteria optimization}}.
\newblock Springer Berlin Heidelberg New York, 2 edition, 2005.

\bibitem{CLV07}
C.~A. {Coello Coello}, G.~B. Lamont, and D.~A. {Van Veldhuizen}.
\newblock {\em {Evolutionary Algorithms for Solving Multi-Objective Problems}},
  volume~2.
\newblock Springer Science {\&} Business Media, 2007.

\bibitem{GP17}
L.~Gr{\"{u}}ne and J.~Pannek.
\newblock {\em {Nonlinear Model Predictive Control}}.
\newblock Springer International Publishing, 2 edition, 2017.

\bibitem{PD18}
S.~Peitz and M.~Dellnitz.
\newblock {A Survey of Recent Trends in Multiobjective Optimal Control --
  Surrogate Models, Feedback Control and Objective Reduction}.
\newblock {\em Mathematical and Computational Applications}, 23(2), 2018.

\bibitem{BP09a}
A.~Bemporad and D.~{Mu{\~{n}}oz de la Pe{\~{n}}a}.
\newblock {Multiobjective model predictive control}.
\newblock {\em Automatica}, 45(12):2823--2830, 2009.

\bibitem{ZFT12}
V.~M. Zavala and A.~Flores-Tlacuahuac.
\newblock {Stability of multiobjective predictive control: A utopia-tracking
  approach}.
\newblock {\em Automatica}, 48(10):2627--2632, 2012.

\bibitem{LBK08}
K.~Laabidi, F.~Bouani, and M.~Ksouri.
\newblock {Multi-criteria optimization in nonlinear predictive control}.
\newblock {\em Mathematics and Computers in Simulation}, 76(5-6):363--374,
  2008.

\bibitem{GGG+12}
J.~J.~V. Garc{\'{i}}a, V.~G. Garay, E.~I. Gordo, F.~A. Fano, and M.~L. Sukia.
\newblock {Intelligent Multi-Objective Nonlinear Model Predictive Control
  (iMO-NMPC): Towards the `on-line' optimization of highly complex control
  problems}.
\newblock {\em Expert Systems with Applications}, 39:6527--6540, 2012.

\bibitem{KWTD11}
M.~Kr{\"{u}}ger, K.~Witting, A.~Tr{\"{a}}chtler, and M.~Dellnitz.
\newblock {Parametric Model-Order Reduction in Hierarchical Multiobjective
  Optimization of Mechatronic Systems}.
\newblock In {\em Proceedings of the 18th IFAC World Congress 2011, Milano,
  Italy}, volume~18, pages 12611--12619. Elsevier Oxford, 2011.

\bibitem{HNS+13}
C.~Hern{\'{a}}ndez, Y.~Naranjani, Y.~Sardahi, W.~Liang, O.~Sch{\"{u}}tze, and
  J.~Q. Sun.
\newblock {Simple cell mapping method for multi-objective optimal feedback
  control design}.
\newblock {\em International Journal of Dynamics and Control}, 1(3):231--238,
  2013.

\bibitem{GS17}
L.~Gr{\"{u}}ne and M.~Stieler.
\newblock {Performance Guarantees for Multiobjective Model Predictive Control}.
\newblock In {\em IEEE 56th Annual Conference on Decision and Control (CDC)},
  pages 5545--5550, 2017.

\bibitem{PSOB+17}
S.~Peitz, K.~Sch{\"{a}}fer, S.~Ober-Bl{\"{o}}baum, J.~Eckstein,
  U.~K{\"{o}}hler, and M.~Dellnitz.
\newblock {A Multiobjective MPC Approach for Autonomously Driven Electric
  Vehicles}.
\newblock {\em IFAC PapersOnLine}, 50(1):8674--8679, 2017.

\bibitem{BMDP02}
A.~Bemporad, M.~Morari, V.~Dua, and E.~N. Pistikopoulos.
\newblock {The explicit linear quadratic regulator for constrained systems}.
\newblock {\em Automatica}, 38(1):3--20, 2002.

\bibitem{Frazzoli2001}
E.~Frazzoli.
\newblock {\em Robust Hybrid Control for Autonomous Vehicle Motion Planning}.
\newblock PhD thesis, Massachusetts Institute of Technology, 2001.

\bibitem{FrDaFe05}
E.~Frazzoli, M.~A. Dahleh, and E.~Feron.
\newblock Maneuver-based motion planning for nonlinear systems with symmetries.
\newblock {\em IEEE Transactions on Robotics}, 21(6):1077--1091, 2005.

\bibitem{Kob08}
M.~Kobilarov.
\newblock {\em {Discrete Geometric Motion Control of Autonomous Vehicles}}.
\newblock PhD thesis, University of Southern California, 2008.

\bibitem{FOK10}
K.~Fla{\ss}kamp, S.~Ober-Bl{\"{o}}baum, and M.~Kobilarov.
\newblock {Solving optimal control problems by using inherent dynamical
  properties}.
\newblock {\em Proceedings of Applied Mathematics and Mechanics},
  10(1):577--578, 2010.

\bibitem{DB12}
C.~Danielson and F.~Borrelli.
\newblock {\em {Symmetric Explicit Model Predictive Control}}, volume~4.
\newblock IFAC, 2012.

\bibitem{OJM11}
S.~Ober-Bl{\"{o}}baum, O.~Junge, and J.~E. Marsden.
\newblock {Discrete mechanics and optimal control: an analysis}.
\newblock {\em Control, Optimisation and Calculus of Variations},
  17(2):322--352, 2011.

\bibitem{Lib12}
D.~Liberzon.
\newblock {\em {Calculus of Variations and Optimal Control Theory: A Concise
  Introduction}}.
\newblock Princeton University Press, 2012.

\bibitem{Betts98}
J.~T.\ Betts.
\newblock Survey of numerical methods for trajectory optimization.
\newblock {\em AIAA Journal of Guidance, Control, and Dynamics},
  21(2):193--207, 1998.

\bibitem{ORZG12}
S.~Ober-Bl\"obaum, M.~Ringkamp, and G.~zum Felde.
\newblock Solving multiobjective optimal control problems in space mission
  design using discrete mechanics and reference point techniques.
\newblock In {\em 51st IEEE International Conference on Decision and Control},
  pages 5711--5716, Maui, HI, USA, 10-13 December 2012.

\bibitem{LSKV10}
F.~Logist, S.~Sager, C.~Kirches, and J.F. {Van Impe}.
\newblock Efficient multiple objective optimal control of dynamic systems with
  integer controls.
\newblock {\em Journal of Process Control}, 20(7):810 -- 822, 2010.

\bibitem{SWO+13}
O.~Sch{\"{u}}tze, K.~Witting, S.~Ober-Bl{\"{o}}baum, and M.~Dellnitz.
\newblock {Set Oriented Methods for the Numerical Treatment of Multiobjective
  Optimization Problems}.
\newblock In E.~Tantar, A.-A. Tantar, P.~Bouvry, P.~{Del Moral}, P.~Legrand,
  C.~A. {Coello Coello}, and O.~Sch{\"{u}}tze, editors, {\em EVOLVE - A Bridge
  between Probability, Set Oriented Numerics and Evolutionary Computation},
  volume 447 of {\em Studies in Computational Intelligence}, pages 187--219.
  Springer Berlin Heidelberg, 2013.

\bibitem{Hil01}
C.~Hillermeier.
\newblock {\em {Nonlinear Multiobjective Optimization: A Generalized Homotopy
  Approach}}.
\newblock Birkh{\"{a}}user, 2001.

\bibitem{DSH05}
M.~Dellnitz, O.~Sch{\"{u}}tze, and T.~Hestermeyer.
\newblock {Covering Pareto sets by Multilevel Subdivision Techniques}.
\newblock {\em Journal of Optimization Theory and Applications},
  124(1):113--136, 2005.

\bibitem{Pei17}
S.~Peitz.
\newblock {\em {Exploiting Structure in Multiobjective Optimization and Optimal
  Control}}.
\newblock PhD thesis, Paderborn University, 2017.

\bibitem{RA09}
J.~B. Rawlings and R.~Amrit.
\newblock {Optimizing process economic performance using model predictive
  control}.
\newblock In {\em Nonlinear model predictive control}, pages 119--138. Springer
  Berlin Heidelberg, 2009.

\bibitem{DAR11}
M.~Diehl, R.~Amrit, and J.~B. Rawlings.
\newblock {A Lyapunov Function for Periodic Economic Optimizing Model
  Predictive Control}.
\newblock {\em IEEE Transactions on Automatic Control}, 56(3):703--707, 2011.

\bibitem{AB09}
A.~Alessio and A.~Bemporad.
\newblock {A survey on explicit model predictive control}.
\newblock In L.~Magni, D.~M. Raimondo, and F.~Allg{\"{o}}wer, editors, {\em
  Nonlinear Model Predictive Control: Towards New Challenging Applications},
  volume 384, pages 345--369. Springer Berlin Heidelberg, 2009.

\bibitem{Zav15}
V.~M. Zavala.
\newblock {A Multiobjective Optimization Perspective on the Stability of
  Economic MPC}.
\newblock {\em IFAC-PapersOnLine}, 48(8):974--980, 2015.

\bibitem{MMC11}
J.~M. Maester, D.~{Mu{\~{n}}oz de la Pe{\~{n}}a}, and E.~F. Camacho.
\newblock {Distributed model predictive control based on a cooperative game}.
\newblock {\em Optimal Control Applications and Methods}, 32:153--176, 2011.

\bibitem{NCS+14}
A.~N{\'{u}}{\~{n}}ez, C.~E. Cort{\'{e}}s, D.~S{\'{a}}ez, B.~{De Schutter}, and
  M.~Gendreau.
\newblock {Multiobjective model predictive control for dynamic pickup and
  delivery problems}.
\newblock {\em Control Engineering Practice}, 32:73--86, 2014.

\bibitem{BuLe04}
F.~Bullo and A.~D. Lewis.
\newblock {\em Geometric Control of Mechanical Systems}, volume~49 of {\em
  Texts in Applied Mathematics}.
\newblock Springer, 2004.

\bibitem{Bloch}
A.~M. Bloch.
\newblock {\em Nonholonomic mechanics and control}.
\newblock Springer, 2003.

\bibitem{FrBu02}
E.~Frazzoli and F.~Bullo.
\newblock On quantization and optimal control of dynamical systems with
  symmetries.
\newblock In {\em {Proceedings of the 41st IEEE Conference on Decision and
  Control}}, pages 817--823, 2002.

\bibitem{FO12-ECMI}
K.~Fla{\ss}kamp and S.~{Ober-Bl\"obaum}.
\newblock Motion planning for mechanical systems with hybrid dynamics.
\newblock In M.~Fontes, M.~G\"unther, and N.~Marheineke, editors, {\em Progress
  in Industrial Mathematics at ECMI 2012}, volume~19 of {\em The European
  Consortium for Mathematics in Industry}. Springer, 2014.

\bibitem{FHO15}
K.~Fla{\ss}kamp, S.~{Hage-Packh\"auser}, and S.~{Ober-Bl\"obaum}.
\newblock Symmetry exploiting control of hybrid mechanical systems.
\newblock {\em Journal of Computational Dynamics}, 2(1):25--50, 2015.

\bibitem{Wit12}
K.~Witting.
\newblock {\em {Numerical algorithms for the treatment of parametric
  multiobjective optimization problems and applications}}.
\newblock PhD thesis, University of Paderborn, 2012.

\bibitem{Joh02}
T.~A. Johansen.
\newblock {On multi-parametric nonlinear programming and explicit nonlinear
  model predictive control}.
\newblock In {\em 41st IEEE Conference on Decision and Control}, volume~3,
  pages 2768--2773, 2002.

\bibitem{BF06}
A.~Bemporad and C.~Filippi.
\newblock {An Algorithm for Approximate Multiparametric Convex Programming}.
\newblock {\em Computational Optimization and Applications}, 35(1):87--108,
  2006.

\bibitem{PCY+16}
B.~Paden, M.~\v{C}\'{a}p, S.~Z. Yong, D.~Yershov, and E.~Frazzoli.
\newblock A survey of motion planning and control techniques for self-driving
  urban vehicles.
\newblock {\em IEEE Transactions on Intelligent Vehicles}, 1(1):33--55, 2016.

\bibitem{LLRW11}
S.~Li, K.~Li, R.~Rajamani, and J.~Wang.
\newblock {Model Predictive Multi-Objective Vehicular Adaptive Cruise Control}.
\newblock {\em IEEE Transactions on Control Systems Technology},
  19(3):556--566, 2011.

\bibitem{TL90}
S.~Taheri and E.~H. Law.
\newblock {Investigation of a combined slip control braking and closed loop
  four wheel steering system for an automobile during combined hard braking and
  severe steering}.
\newblock {\em American Control Conference}, pages 1862--1867, 1990.

\bibitem{PL14}
G.~Perantoni and D.~J.~N. Limebeer.
\newblock {Optimal control for a Formula One car with variable parameters}.
\newblock {\em Vehicle System Dynamics}, 52(5):653--678, 2014.

\bibitem{DEF+14}
M.~Dellnitz, J.~Eckstein, K.~Fla{\ss}kamp, P.~Friedel, C.~Horenkamp,
  U.~K{\"{o}}hler, S.~Ober-Bl{\"{o}}baum, S.~Peitz, and S.~Tiemeyer.
\newblock {Development of an Intelligent Cruise Control Using Optimal Control
  Methods}.
\newblock {\em Procedia Technology}, 15:285--294, 2014.

\bibitem{DEF+17}
M.~Dellnitz, J.~Eckstein, K.~Fla{\ss}kamp, P.~Friedel, C.~Horenkamp,
  U.~K{\"{o}}hler, S.~Ober-Bl{\"{o}}baum, S.~Peitz, and S.~Tiemeyer.
\newblock {Multiobjective Optimal Control Methods for the Development of an
  Intelligent Cruise Control}.
\newblock In G.~Russo, V.~Capasso, G.~Nicosia, and V.~Romano, editors, {\em
  Progress in Industrial Mathematics at ECMI 2014}, pages 633--641. Springer,
  2017.

\bibitem{EPS+16}
J.~Eckstein, S.~Peitz, K.~Sch{\"{a}}fer, P.~Friedel, U.~K{\"{o}}hler,
  M.~{Hessel-von Molo}, S.~Ober-Bl{\"{o}}baum, and M.~Dellnitz.
\newblock {A Comparison of two Predictive Approaches to Control the
  Longitudinal Dynamics of Electric Vehicles}.
\newblock {\em Procedia Technology}, 26:465--472, 2016.

\end{thebibliography}
}
\end{document}